\documentclass[a4paper]{amsart}
\usepackage[utf8]{inputenc}

\usepackage{enumitem}

\usepackage{amsmath}
\usepackage{mathtools}
\usepackage{amssymb}
\usepackage{amsfonts}

\usepackage{microtype}

\usepackage{xcolor}

\usepackage{tikz}
\usetikzlibrary{arrows.meta,bending}
\usepackage{tkz-euclide}
\usetikzlibrary{positioning}
\usetikzlibrary{patterns}

\allowdisplaybreaks

\usepackage{mathtools}

\newcommand{\natNum}{\mathbb{N}}
\newcommand{\realNum}{\mathbb{R}}
\newcommand{\complNum}{\mathbb{C}}
\newcommand{\posRealNum}{[0,\infty)}
\newcommand{\semiGroup}[1]{\mathbb{T}(#1)}
\newcommand{\semiGroupDef}{\mathbb{T} = \left( \semiGroup{t} \right)_{t \geq 0}}
\newcommand{\loc}{\mathrm{loc}}
\newcommand{\andMath}{\textrm{and}}
\newcommand{\realPart}[1]{\mathrm{Re}\left( #1 \right)}
\newcommand{\set}[1]{\left\lbrace #1 \right\rbrace}
\newcommand{\innerProd}[2]{\left\langle #1 , #2 \right\rangle}
\newcommand{\lpSpaceDT}[3]{L^{#1}\left(#2, #3\right)}
\newcommand{\lpSpacelocDT}[3]{L^{#1}_{\loc}\left(#2, #3\right)}
\newcommand{\laplaceTr}[1]{\mathcal{L}\left\lbrace #1 \right\rbrace}
\newcommand{\laplaceInvTr}[1]{\mathcal{L}^{-1}\left\lbrace #1 \right\rbrace}
\newcommand{\intd}[1]{\,\mathrm{d}#1}
\newcommand{\der}[2]{\frac{\mathrm{d}#1}{\mathrm{d}#2}}
\newcommand{\derOrd}[3]{\frac{\mathrm{d}^{#3}#1}{\mathrm{d}#2^{#3}}}
\newcommand{\partialDer}[2]{\frac{\partial #1}{\partial #2}}

\newcommand{\dom}{\mathrm{D}}
\newcommand{\pH}[6]{$\Sigma_{#1}\left(#2, #3, #4, #5, #6 \right)$}
\newcommand{\pHSmallBrack}[6]{$\Sigma_{#1}(#2, #3, #4, #5, #6 )$}
\newcommand{\pHstand}{\pHSmallBrack{n}{P_1}{P_0}{\mathcal{H}}{\widetilde{W}_B}{\widetilde{W}_C}}

\usepackage{bbm}
\newcommand{\idMatrix}[1]{\mathbbm{1}_{#1}}
\newcommand{\absMatrix}[1]{#1_{|\cdot|}}

\newcommand{\matrixNorm}[3]{\| #1 \|_{\ell^{#2} \rightarrow \ell^{#3}}}

\newtheorem{thrm}{Theorem}
\newtheorem{prpstn}[thrm]{Proposition}
\newtheorem{lmm}[thrm]{Lemma}
\newtheorem{crllr}[thrm]{Corollary}
\newtheorem{cnjctr}[thrm]{Conjecture}

\theoremstyle{definition}
\newtheorem{dfntn}[thrm]{Definition}
\newtheorem{xmpl}[thrm]{Example}
\newtheorem{assumption}[thrm]{Assumption}

\theoremstyle{remark}
\newtheorem{rmrk}[thrm]{Remark}

\begin{document}


\title{BIBO stability of 1-D hyperbolic boundary control systems}
\thanks{Wierzba is supported by the Theme Team project ``Predictive Avatar Control and Feedback'' sponsored by the Faculty of Eletrical Engineering, Computer Science and Mathematics.}
\author{Felix L.~Schwenninger}
\author[Alexander A. Wierzba]{Alexander A. Wierzba*}
\thanks{*Corresponding author}
\address{Department of Applied Mathematics, University of Twente, P.O. Box 217, 7500 AE Enschede, The Netherlands}\email{f.l.schwenninger@utwente.nl, alexander.wierzba@uni-wuppertal.de}
\begin{abstract}
We study the question of bounded-input bounded-output (BIBO) stability of a class of 1-D hyperbolic boundary control systems, which, in particular, contains distributed port-Hamiltonian systems. Exploiting the particular structure of the transfer function of these systems, we derive several sufficient conditions for BIBO stability.
\end{abstract}

\maketitle

%


\section{Introduction}
Spatially one-dimensional, hyperbolic evolution equations, such as given by port-Hamiltonian systems (pHS) \cite{a_vanDerSchaftMaschke02DistributedPHS, b_JacobZwart2012, a_JacobZwart23_InfiniteDimPHS} or hyperbolic balance laws \cite{b_BastinCoron16_Stability1dHyperbolicSystems, b_Dafermos16_HyperbolicConservationLawsInContPhys, b_Serre99_SystemsOfConservationLaws, b_Bartecki16_ModelingAnalysisOfLinHypSystemsBalanceLaws} have constituted an important system class used in modelling and control of real world applications. Their close connection to the ideas of energy-flow and energy-conservation makes them a useful framework for the study of physical systems\cite{a_vanDerSchaft06pHSIntroduction, a_vanDerSchaft20PHSModeling, b_Duindam09RedPHSBook} both mathematically as well as from an engineering perspective.

A particular type of example are systems of port-Hamiltonian type, as originally defined in \cite{b_JacobZwart2012, th_Villegas2007, a_leGorrecZwartMaschke05DiracStructuresBCS, a_JacobZwart18OperatorTheoreticApproach}, and including those being impedance passive \cite{th_Augner2016,a_JacobZwart23_InfiniteDimPHS, a_vanDerSchaft06pHSIntroduction}, which nowadays are  commonly referred to as \emph{port-Hamiltonian systems}. These include systems given by a PDE of the form
\begin{equation*}
    \partialDer{x}{t} = P_1 \partialDer{}{\xi} \left( \mathcal{H} x \right) + P_0 \left( \mathcal{H} x \right)
\end{equation*}
on some spatial interval $[a,b]$ together with boundary control and observation at its endpoints. The justification and motivation for the study of theses systems lies in their ability to provide a modelling framework for a variety of different real-world processes, including first of all transport and flow phenomena such as in fluid mechanics \cite{a_PasumarthyVanDerschaft06CanalSystems} or chemistry \cite{a_ZhouEtAl12DistPHSChemicalReactors, a_ZhouEtAl15DistPHSReactionDiffusion}  and also vibration effects within flexible structures as appearing e.g.\ in robotics \cite{th_Mattioni2021,a_MacchelliStramigioliMelchiorri06ManipulatorsFlexibleLinks}. Particularly, in this latter case, the approach using infinite-dimensional systems allows to refine finite-dimensional models where the mentioned phenomena (vibrations, flows, etc.) are usually ignored or only treated in a very simplified way. We refer the interested reader to the literature review \cite{a_RashadEtAl20DistributedPHS} for a more history, and applications of distributed pHS. On the other hand, 1D hyperbolic systems have been extensively studied in the past decades, driven by numerous occurence of transport phenomena in the sciences, \cite{a_BastinCoronHayat21ISSwrtSupNorms,b_BastinCoron16_Stability1dHyperbolicSystems,a_Hayat21BoundaryStabilization}, such as in (traffic) networks, flows in fluid channels and light propagation in optics.  In particular, stabilizing such systems is of key relevance both theoretically and practically. 

A natural and well-studied question for hyperbolic balance laws is whether a system is stable or, as often in a control context,  whether it is stabilizable, see e.g. \ \cite{a_BastinCoronHayat21ISSwrtSupNorms, b_BastinCoron16_Stability1dHyperbolicSystems} and the references therein. Only recently, work has been undertaken to systematically classify different notions for the above-mentioned systems of port-Hamiltonian type. These studies include \emph{internal stability}, and in particular recent characterizations of exponential \cite{a_TrostorffWaurick23ExponentialStabilityPHS}, asymptotic stability \cite{a_WaurickZwart23AsymptoticStabilityPHS} and semi-uniform stability \cite{a_AroraEtAl24SemiUniformStability}. Previously only sufficient conditions for these properties were known \cite{b_JacobZwart2012, th_Villegas2007}, which were however mainly restricted to the case of systems featuring a contraction semigroup as encountered in the vast majority of concrete examples.

On the other hand, the input-output behaviour of boundary control systems is an important aspect for control applications. The systems considered here fall -- under practically weak technical assumptions -- in the class of $L^2$-well-posed systems, see e.g.\ \cite{b_JacobZwart2012}. Thus there exist constants $c_t$ for any $t>0$ such that for any input $u \in \lpSpaceDT{2}{[0,t]}{U}$ the corresponding output $y$ satisfies $\| y \|_{\lpSpaceDT{2}{[0,t]}{Y}} \leq c_t \| u \|_{\lpSpaceDT{2}{[0,t]}{U}}$.
The importance of result like this is grounded in the port-Hamiltonian structure being used to model energy-flows and the input-output stability thus being related to the concepts of energy-conserving and energy-dissipating systems. 

In this contribution we study the  notion of \emph{bounded-input bounded-output stability} (BIBO stability), which is classical for finite-dimensional linear systems, for a class of 1-D hyperbolic boundary control systems. An input-output system is called BIBO stable if any bounded input is mapped to a bounded output and there exists a uniform relative norm bound \cite{a_UnserNoteOnBiboStability, a_CallierDesoer1978, a_WangCobb96BIBOTimeInv, a_WangCobb03BIBOTimeVar, a_AbusaksakaPartingtonBIBO2013, a_SchwenningerWierzbaZwart24BIBO, a_HastirEtAl23NonlinearBIBO}. We thus aim at input-output relations of the form $\| y \|_{\lpSpaceDT{\infty}{[0,t]}{Y}} \leq c \| u \|_{\lpSpaceDT{\infty}{[0,t]}{U}}$ for some positive constant $c$. 

The question of whether a system is BIBO stable or not, while relatively straightforward in the finite-dimensional setting, becomes significantly more involved in the case of an infinite-dimensional state space and control/observation acting on the spatial domain's boundary. For instance, it is then no longer the case that BIBO stability is always implied by exponential stability (see \cite[Thm.~5.1]{a_SchwenningerWierzbaZwart24BIBO}, but note that crucially this example is not hyperbolic). In \cite{a_SchwenningerWierzbaZwart24BIBO} it was shown that, for a linear system described in the framework of system nodes (see \cite{b_Staffans2005}) with finite-dimensional input and output spaces -- as will be the case in the system class considered in this contribution -- BIBO stability is equivalent to the inverse Laplace transform of the transfer function $\mathbf{G}(s)$ being a measure of bounded total variation. This includes in particular the case that the impulse response -- upon being well-defined -- of the system is an $L^1$ function. While this condition offers a complete characterisation of when a system is BIBO stable, it is often of limited applicability for concrete systems. Typically no closed form expression for the inverse Laplace transform of the respective transfer function exists or is prohibitively hard to find. Thus BIBO stability is in general much more difficult to establish than e.g.\ $L^2$-well-posedness. This circumstance motivates the search for other sufficient or necessary conditions and the concrete study of BIBO stability for particular system classes. A few such results have been provided in for example \cite{a_SchwenningerWierzbaZwart24BIBO} already, however almost exclusively for parabolic systems, which exclude the hyperbolic systems considered here.

We note that BIBO stability -- besides being an intriguing topic in itself and classical control theory -- plays an essential role in e.g.\ the applicability of funnel control to relative degree systems \cite{a_IlchmannRyanTrennFunnel2005,a_BergerPucheSchwenningerFunnel2020} and more generally for systems for which high-gain control is used. Such systems with interior dynamics modelled by port-Hamiltonian systems appear for example in the study of mechanical systems containing flexible and deformable parts \cite{th_Mattioni2021}, which significantly influence the behaviour.

This paper begins with the definition of the considered class of hyperbolic 1-D boundary control systems, which is based on the original framework used for port-Hamiltonian systems in \cite{b_JacobZwart2012,a_ZwartGorrecMaschkeVillegas10WellPosedness} in Section \ref{sec:pHS}. 
Then, in Section~\ref{sec:approachOneTransferFunction}, we derive an abstract decomposition of the transfer function , which takes inspiration from recent results for exponential stability of pHS \cite{a_TrostorffWaurick23ExponentialStabilityPHS} and exploit its particular structure in order to apply the known characterisation of BIBO stability in a general setting. This yields a collection of sufficient criteria expressed in the form of matrix conditions for a subclass of pHS. 
Section~\ref{sec:approachTwoWellPosedness} then presents a different approach to BIBO without any reference to transfer functions, instead employing the $L^1$-well-posedness of pHS on an $L^1$ state space and the connection between $L^1$ and $L^\infty$ stability notions discussed in \cite{a_SchwenningerWierzba24_ADualNotionToBIBOStability}.

\subsection{Notation}
For any $\alpha \in \realNum$, let $\complNum_\alpha := \left\lbrace z \in \complNum \;\middle|\; \realPart{z} > \alpha \right\rbrace$. Let $\mathbb{K}$ denote either $\realNum$ or $\complNum$. The Banach spaces considered are over the field $\mathbb{K}$.  

Let $A$ be a linear operator defined on some Banach space $X$. Then we denote by $\dom(A) \subseteq X$ its domain of definition, by $\rho(A)$ its resolvent set and by $\sigma(A)$ its spectrum.
If $A$ is further the generator of a strongly continuous semigroup $\semiGroupDef$ on $X$, then $X_1$ is defined as the space $\dom(A)$ with the norm $\|x\|_{X_1} := \| (\beta I - A) x \|_X$ with $\beta \in \rho(A)$ and $X_{-1}$ as the completion of $X$ with respect to the norm $\|x\|_{X_{-1}} := \| (\beta I - A)^{-1} x \|_X$ again with $\beta \in \rho(A)$. There exists a unique extension $\mathbb{T}_{-1} = \left( \mathbb{T}_{-1}(t) \right)_{t \geq 0}$ of the semigroup $\mathbb{T}$ to the space $X_{-1}$ with generator $A_{-1}:X \rightarrow X_{-1}$ which is an extension of the operator $A$. For more details see \cite{b_Staffans2005,b_TucsnakWeiss2009}. 

Let $\idMatrix{n} \in \mathbb{K}^{n \times n}$ denote the identity matrix. For a matrix $M \in \mathbb{K}^{n \times n}$ and $p,q \in [1,\infty]$ let $\matrixNorm{M}{p}{q}$ be the matrix norm induced by the $\ell^p$-norms on $\mathbb{K}^n$ given by $\| x \|_{\ell^p} = \left( \sum_{k=1}^n |x_k|^p \right)^{\frac{1}{p}}$ for $p \in [1,\infty)$ and $\| x \|_{\ell^\infty} = \sup_{1 \leq k \leq n} |x_k|$ for $p = \infty$ as 
\begin{equation*}
    \matrixNorm{M}{p}{q} = \sup_{x \in \mathbb{K}} \frac{\| M x \|_{\ell^q}}{\| x \|_{\ell^p}}.
\end{equation*}

For a Borel measure $h$ on $[0,\infty)$, let $\laplaceTr{h}$ denote its Laplace transform if it exists (on some right half-plane). For a function $g$ of one complex variable, if there exists a measure $\mu$ such that $\laplaceTr{\mu}$ and $g$ agree on some right half plane we will call $\mu$ the inverse Laplace transform and write $\mu = \laplaceInvTr{g}$.

Let $\mathcal{M}(\posRealNum, \mathbb{K}^{n \times m})$ denote the set of Borel measures of bounded total variation on $\posRealNum$ with values in $\mathbb{K}^{n \times m}$. 
Furthermore, for $h \in \mathcal{M}(\posRealNum, \mathbb{K}^{n \times m})$ let $\| h \|_{\mathcal{M}}$ denote the total variation of $h$ \cite[Sec.~3.2]{b_GripenbergLondenStaffans1990}.

Furthermore, let $\absMatrix{M} \in \mathbb{K}^{n \times n}$ be defined as the matrix with entries $( \absMatrix{M})_{ij} = |M_{ij}|$.

\section{A class of hyperbolic boundary control systems}
\label{sec:pHS}
In this work we consider \emph{linear hyperbolic boundary control systems} of the form
\begin{equation}
\label{eq:pHSequation}
\begin{aligned}
    \partialDer{x}{t}(\xi, t) &= P_1 \partialDer{}{\xi} \left( \mathcal{H}(\xi) x(\xi, t) \right) + P_0(\xi) \left( \mathcal{H}(\xi) x(\xi, t) \right), & (\xi,t) &\in [a,b] \times \posRealNum, \\
    x(\xi, 0) &= x_0(\xi), & \xi &\in [a,b], \\
    u(t) &= \widetilde{W}_B \begin{bmatrix}
        \left( \mathcal{H}x \right)(b,t) \\ \left( \mathcal{H}x \right)(a,t) 
    \end{bmatrix},  & t &\in \posRealNum, \\
    y(t) &= \widetilde{W}_C \begin{bmatrix}
        \left( \mathcal{H}x \right)(b,t) \\ \left( \mathcal{H}x \right)(a,t) 
    \end{bmatrix},  & t &\in \posRealNum,
\end{aligned}
\end{equation}
with the following assumptions:
\begin{assumption}
    \label{ass:systemClassAssumptions}
    \begin{enumerate}
        \item $P_1 \in \mathbb{K}^{n\times n}$ is invertible and self-adjoint;
    \item $P_0 \in \lpSpaceDT{\infty}{[a,b]}{\mathbb{K}^{n\times n}}$;
    \item $\mathcal{H} \in \lpSpaceDT{\infty}{[a,b]}{\mathbb{K}^{n\times n}}$, $\mathcal{H}(\xi)$ is self-adjoint for almost all $\xi \in [a,b]$ and $\exists M,m > 0$ s.t. $m \idMatrix{n} \leq \mathcal{H}(\xi) \leq M \idMatrix{n}$ for almost all $\xi \in [a,b]$;
    \item $\widetilde{W}_B \in \mathbb{K}^{n \times 2n}$ and $\left[ \begin{smallmatrix}
        \widetilde{W}_B \\ \widetilde{W}_C
    \end{smallmatrix} \right] \in \mathbb{K}^{2n \times 2n}$ have full rank;
    \item \label{ass:systemClassAssumptions:semigroup}  the operator $A := P_1 \partialDer{}{\xi}  \mathcal{H} + P_0  \mathcal{H}$ with domain
\begin{align*}
    \dom(A) = \set{x_0 \in X \, \middle| \, \mathcal{H} x_0 \in H^1\left( [a,b], \mathbb{K}^n \right), \widetilde{W}_B \begin{bmatrix}
        \left( \mathcal{H}x_0 \right)(b) \\ \left( \mathcal{H}x_0 \right)(a) 
    \end{bmatrix} = 0 }
    \end{align*}
    is the infinitesimal generator of a $C_0$-semigroup on the state space $X = \lpSpaceDT{2}{[a,b]}{\mathbb{K}^n}$ with inner product \[\innerProd{f}{g}_X := \frac{1}{2} \int_a^b g(\xi)^* \mathcal{H}(\xi) f(\xi) \intd \xi.\]
    \end{enumerate}
\end{assumption}
We will in the following use \pHstand{} to denote this hyperbolic boundary control system (hBCS). 

Note that a similar class was studied in the seminal work \cite{a_leGorrecZwartMaschke05DiracStructuresBCS,a_ZwartGorrecMaschkeVillegas10WellPosedness} in the context of $L^2$-well-posedness, which marked the beginning of what later became known as distributed port-Hamiltonian systems \cite{b_JacobZwart2012}. We emphasise that the class studied here particularly includes these port-Hamiltonian systems, both with \cite{th_Augner2016, a_JacobZwart23_InfiniteDimPHS} or without \cite{b_JacobZwart2012} assumed impedance passivity, which in the meantime seems to be commonly assumed for port-Hamiltonian systems. Our notation used here originates from the text book \cite{b_JacobZwart2012} and as there we will refer to $\mathcal{H}$ as the \emph{Hamiltonian} of the system.

\begin{rmrk}
\label{rem:assumptionRemarks}
    \begin{enumerate}
    \item Note in particular, that we do not assume at this point that $A$ generates a \emph{contraction} semigroup.
    \item Since $A$ is a bounded perturbation of the operator $\frac{\partial}{\partial \xi} \mathcal{H}$, with the same domain, the assumption that $A$ generates a $C_{0}$-semigroup,  Assumption \ref{ass:systemClassAssumptions}(5),  is independent of the choice of $P_{0}$.
        \item The Assumption \ref{ass:systemClassAssumptions}(5) of $A$ generating a $C_0$-semigroup is in particular satisfied if 
        \begin{align}
            \label{eq:contractionCondition}
            \widetilde{W}_B R_0^{-1} \Sigma \left( \widetilde{W}_B R_0^{-1} \right)^* \geq 0
        \end{align}
        where $R_0 = \frac{1}{\sqrt{2}} \left[ \begin{smallmatrix}
        P_1 & - P_1 \\ \idMatrix{n} & \idMatrix{n}
    \end{smallmatrix} \right]$ and $\Sigma = \left[ \begin{smallmatrix}
        0 & \idMatrix{n} \\ \idMatrix{n} & 0
\end{smallmatrix} \right]$. If in addition $P_0^\ast = - P_0$, then \eqref{eq:contractionCondition} is even equivalent to the semigroup being contractive \cite[Thm.~7.2.4, Thm.~10.3.1]{b_JacobZwart2012}.
     \end{enumerate}
\end{rmrk}

\noindent For our system class, direct extension of \cite[Thm.~11.3.2~\&~Thm.~12.1.3]{b_JacobZwart2012} to the case of spatially dependent $P_0$, shows that the \emph{transfer function} $\mathbf{G}$ can be found for $s \in \rho(A)$ as the unique solution of
\begin{equation}
\begin{split}
    \label{eq:transferFunctionODE}
        &s x_0(\xi) = P_1 \der{}{\xi} \left( \mathcal{H}(\xi) x_0(\xi) \right) + P_0(\xi) \left( \mathcal{H}(\xi) x_0(\xi) \right), \\
        &u_0 = \widetilde{W}_B \begin{bmatrix}
        \left( \mathcal{H}x \right)(b) \\ \left( \mathcal{H}x \right)(a) 
    \end{bmatrix}, \qquad
    \mathbf{G}(s) u_0 = \widetilde{W}_C \begin{bmatrix}
        \left( \mathcal{H}x \right)(b) \\ \left( \mathcal{H}x \right)(a) 
    \end{bmatrix}.
\end{split}
\end{equation}

\subsection{BIBO stability of hyperbolic boundary control systems}
In this paper we want to study the question of BIBO stability of systems \eqref{eq:pHSequation} satisfying Assumption~\ref{ass:systemClassAssumptions}.
 In the following we will understand the term \emph{classical solution} in the common sense of evolution equations, that is as a collection $(u,x,y)$ with $u \in C(\posRealNum, U)$, $x \in C^1(\posRealNum, X)$ and $y \in C(\posRealNum,Y)$ solving Equation~\eqref{eq:pHSequation}.
\begin{dfntn}
\label{def:BIBOStabClassSolutions}
A hBCS is called \emph{BIBO stable} if there exists $c > 0$ such that for any of its classical solution $(u,x,y)$ with $u \in C^\infty_c(\posRealNum, U)$ and $x(0) = 0$ and for all $t > 0$ we have that
$\| y \|_{\lpSpaceDT{\infty}{[0,t]}{Y}} \leq c \| u \|_{\lpSpaceDT{\infty}{[0,t]}{U}}$.
\end{dfntn}

\begin{rmrk}
Any hBCS satisfying Assumption~\ref{ass:systemClassAssumptions} can be equivalently expressed within the framework of system nodes (see \cite{b_Staffans2005} for a general introduction and \cite[Ch.~11]{b_JacobZwart2012}\cite{a_Schwenninger20ISSboundaryControl} for the rewriting of hBCS respectively general boundary control systems in this form). In particular, the following observations can be made.
    \begin{enumerate}
        \item 
        Following the construction in \cite{b_Staffans2005}, we can define a generalised solution $(u,x,y)$ of a hBCS for any input function $u \in \lpSpacelocDT{1}{\posRealNum}{U}$ and initial value $x_0 \in X$. 
        
        Here the state is given by $x(t) = \semiGroup{t} x_0 + \int_0^t \mathbb{T}_{-1}\left( t - s \right) B u(s) \intd{s} \in X_{-1}$, where $\mathbb{T}$ is the semigroup generated by the operator $A$ from Assumption~\ref{ass:systemClassAssumptions}(5), and the output in the sense of distributions as
        $
    y(t) = \derOrd{}{t}{2} \widetilde{W}_C \left[ \begin{smallmatrix}
        \left( \mathcal{H}z \right)(b,t) \\ \left( \mathcal{H}z \right)(a,s) \end{smallmatrix} 
    \right]
    $
    where $z(t) := \int_0^t (t - s) x(t) \intd s$.
    \item The hBCS considered here fall within the class of systems investigated for BIBO stability in~\cite{a_SchwenningerWierzbaZwart24BIBO}. In particular, by \cite[Thm.~3.5]{a_SchwenningerWierzbaZwart24BIBO}, any generalised solution $(u,x,y)$ of a BIBO stable hBCS with $u \in \lpSpacelocDT{\infty}{\posRealNum}{U}$ satisfies $y \in \lpSpacelocDT{\infty}{\posRealNum}{Y}$ (instead of only being a distribution) and for any $t > 0$ we have that $\| y \|_{\lpSpaceDT{\infty}{[0,t]}{Y}} \leq c \| u \|_{\lpSpaceDT{\infty}{[0,t]}{U}}$, i.e.\ the inequality from Definition~\ref{def:BIBOStabClassSolutions} extends to generalised solutions with bounded inputs.
    \end{enumerate}
\end{rmrk}

\noindent 
That the question of BIBO stability of the system class considered here is not trivial in either way is shown by the following proposition.
\begin{prpstn}
\label{prop:existenceNonBIBOpHS}
    \begin{enumerate}
        \item There exists a BIBO stable hBCS satisfying Assumption~\ref{ass:systemClassAssumptions}.
        \item There exists a hBCS \pHstand{}  satisfying Assumption~\ref{ass:systemClassAssumptions}, $P_0^* = -P_0$  and such that Equation~\eqref{eq:contractionCondition} holds, which is not BIBO stable. 
    \end{enumerate}
\end{prpstn}
\begin{proof}
    \begin{enumerate}
        \item Take \pH{1}{1}{0}{1}{\left[ \begin{smallmatrix}
        1 & 0
    \end{smallmatrix} \right]}{\left[ \begin{smallmatrix}
        0 & 1
    \end{smallmatrix} \right]}. 
    Then the system is just the trivially BIBO stable transport equation with input $u$ and output $y$   
    \begin{align*}
    \partialDer{x}{t}(\xi, t) &=  \partialDer{}{\xi} \left( x(\xi, t) \right) , &
    x(\xi, 0) &= x_0(\xi),  \\
    u(t) &= x(b,t),   &
    y(t) &= x(a,t).
    \end{align*}    
    \item Take \pH{1}{1}{0}{1}{\left[ \begin{smallmatrix}
        1 & -1
    \end{smallmatrix} \right]}{\left[ \begin{smallmatrix}
        0 & 1
    \end{smallmatrix} \right]}. 
Then $\widetilde{W}_B R_0^{-1} \Sigma \left( \widetilde{W}_B R_0^{-1} \right)^* = 0$ and thus the system satisfies in particular Assumption \ref{ass:systemClassAssumptions}(5) by Remark~\ref{rem:assumptionRemarks} (and the semigroup is even contractive). On the other hand, the constant input $u(t) = 1 \in \lpSpaceDT{\infty}{\posRealNum}{\realNum}$, yields the output $y(t) = \lfloor t \rfloor$ and hence the system is not BIBO stable. \qedhere
    \end{enumerate}
\end{proof}

It turns out that even restricting to a subtype of the system class, i.e.\ adding further assumptions on the boundary control systems, such as impedance passivity (see e.g.\ \cite[Sec.~3.2]{th_Augner2016}), does not improve the situation.
\begin{xmpl}[Impedance passive but non-BIBO stable pHS]
\label{ex:Ch05:impPassiveNotBIBO}
    Consider the system \pH{2}{\idMatrix{2}}{0}{\left[\begin{smallmatrix}
        \frac{1}{2} & 0 \\ 0 & 1
    \end{smallmatrix}\right]}{\widetilde{W}_B}{\widetilde{W}_C} on $[0,1]$ with 
    \begin{align*}
        \widetilde{W}_B = \begin{bmatrix}
            -1 & 0 & \frac{1}{2} & \frac{1}{2} \\ 0 & -1 & -\frac{1}{2} & -\frac{1}{2}
        \end{bmatrix} \quad \andMath \quad \widetilde{W}_C = \frac{1}{2}\begin{bmatrix}
            -1 & -1 & -1 & -1 \\ 1 & -3 & 1 & -3
        \end{bmatrix},
    \end{align*}
    that is the system given by the equations
    \begin{equation*}
        \begin{aligned}
        \partialDer{x_1}{t}(\xi, t) &= \frac{1}{2} \partialDer{x_1}{\xi}(\xi, t), & (\xi,t) &\in [a,b] \times \posRealNum, \\
        \partialDer{x_2}{t}(\xi, t) &= \hphantom{\frac{1}{2}} \partialDer{x_2}{\xi}(\xi, t), & (\xi,t) &\in [a,b] \times \posRealNum, \\
        u_1(t) &= - \frac{1}{2} x_1(1,t) + \frac{1}{4} x_1(0,t) + \frac{1}{2} x_2(0,t), & t &\in \posRealNum, \\
        u_2(t) &= - \hphantom{\frac{1}{2}} x_2(1,t) - \frac{1}{4} x_1(0,t) - \frac{1}{2} x_2(0,t), & t &\in \posRealNum, \\
        y_1(t) &= - \frac{1}{4} x_1(1,t) - \frac{1}{2} x_2(1,t) - \frac{1}{4} x_1(0,t) - \frac{1}{2} x_2(0,t), & t &\in \posRealNum, \\
        y_2(t) &= \hphantom{-} \frac{1}{4} x_1(1,t) - \frac{3}{2} x_2(1,t) + \frac{1}{4} x_1(0,t) - \frac{3}{2} x_2(0,t), & t &\in \posRealNum.
    \end{aligned}
    \end{equation*}
    Using the characterisation from \cite[Prop.~3.2.16]{th_Augner2016} we can show that this is an impedance passive port-Hamiltonian system.

    However, this system is up to the choice of $\widetilde{W}_C$ the same as the one shown to be not BIBO stable in Proposition \ref{prop:BIBOdepHamiltonian}. As $\left[ \begin{smallmatrix}
        \widetilde{W}_B \\ \widetilde{W}_C
    \end{smallmatrix} \right]$ has full rank, we can write the outputs $y_1$ and $y_2$ as linear combinations of the inputs and outputs of that system. Hence this system cannot be BIBO stable either.
\end{xmpl}

\noindent As the system class considered here features finite-dimensional input and output spaces, Theorem 3.4 from \cite{a_SchwenningerWierzbaZwart24BIBO} provides a necessary and sufficient condition for BIBO stability. That is the system is BIBO stable if and only if the transfer function $\mathbf{G}$ is the Laplace transform of a measure of bounded total variation on $\posRealNum$.

While in theory this allows to check any hBCS of this form for BIBO stability, in practice this condition is often not usable. Consider as an example the system from the proof of Proposition \ref{prop:existenceNonBIBOpHS}. Using Theorem 12.2.1 from \cite{b_JacobZwart2012} we find that its transfer function is $\mathbf{G}(s) = \frac{1}{e^s - 1}$.  Already for this simple function no nice closed form expression exists for its inverse Laplace transform, preventing us from a simple application of the criterion. As it turns out, this is the case generically, rendering this approach to proofing respectively disproofing BIBO stability to be of little use in practice if looking at a particular system. However, we will see that the transfer function and the properties of its inverse Laplace transform can -- somewhat surprisingly -- still be used to derive sufficient conditions for subclasses of these systems. 

Finally, note that -- similar as for  exponential stability \cite{a_TrostorffWaurick23ExponentialStabilityPHS} -- no condition only in terms of the matrices $P_1$ and $\widetilde{W}_B$ can fully classify BIBO stability for these systems.
\begin{prpstn}
\label{prop:BIBOdepHamiltonian}
    There exists a BIBO stable \pHSmallBrack{n}{P_1}{P_0}{\mathcal{H}}{\widetilde{W}_B}{\widetilde{W}_C} and a $\widetilde{\mathcal{H}} \in \lpSpaceDT{\infty}{[a,b]}{\mathbb{K}^{n\times n}}$ such that \pHSmallBrack{n}{P_1}{P_0}{\widetilde{\mathcal{H}}}{\widetilde{W}_B}{\widetilde{W}_C} is not BIBO stable.
\end{prpstn}
\begin{proof}
    Consider first \pHSmallBrack{2}{\idMatrix{2}}{0}{\idMatrix{2}}{\widetilde{W}_B}{\widetilde{W}_C} with 
     $\widetilde{W}_B = \left[ \begin{smallmatrix}
            -1 & 0 & \frac{1}{2} & \frac{1}{2} \\ 0 & -1 & -\frac{1}{2} & -\frac{1}{2}    \end{smallmatrix} \right]$ and $ W_C = \left[ \begin{smallmatrix}
            0 & 0 & 1 & 0 \vphantom{\frac{1}{2}} \\
            0 & 0 & 0 & 1 \vphantom{\frac{1}{2}}
    \end{smallmatrix} \right]$.
    Its transfer function is
    $
        \mathbf{G}(s) = \frac{1}{2} \left[ \begin{smallmatrix}
            - e^{-2s} - 2 e^{-s} & - e^{-2s} \\
             e^{-2s} &  e^{-2s} - 2 e^{-s}
        \end{smallmatrix} \right]
    $
    which has the inverse Laplace transform
    $
        \laplaceInvTr{\mathbf{G}}(t) = \frac{1}{2} \left[ \begin{smallmatrix}
            - \delta(t - 2) - 2 \delta(t - 1) & - \delta(t - 2) \\
             \delta(t - 2) &  \delta(t - 2) - 2 \delta(t - 1)
        \end{smallmatrix} \right].
    $
    This is a measure of bounded total variation and hence the system is BIBO stable by \cite[Thm.~3.4]{a_SchwenningerWierzbaZwart24BIBO}.

    Change now the Hamiltonian to arrive at the system \pH{2}{\idMatrix{2}}{0}{\left[\begin{smallmatrix}
        \frac{1}{2} & 0 \\ 0 & 1
    \end{smallmatrix}\right]}{\widetilde{W}_B}{\widetilde{W}_C}
    Let then $u = \left[\begin{smallmatrix}
        u_1 \\ u_2
    \end{smallmatrix}\right]$ be any smooth input such that $u_1(t)=0$ and $u_2(t) = (-1)^{t+1}$ for $t \in \natNum$. Then one can show that $y_2(t) = x_2(0,t) = \frac{1}{9} (-1)^{t+1} \left( 3 t + 4 - 4 \left(-\frac{1}{2} \right)^{t+1} \right)$ for $t \in \natNum$. Thus $\lim_{t\rightarrow \infty} |y_2(t)| = \infty$ and  the system cannot be BIBO stable.
\end{proof}

\section{An approach to BIBO stability via decomposition of the transfer function}
\label{sec:approachOneTransferFunction}
\label{sec:transferFunction}
A first approach to deriving sufficient conditions for BIBO stability is based on applying the characterisation from \cite[Thm.~3.4]{a_SchwenningerWierzbaZwart24BIBO} not in an individual case-by-case manner --- which as we saw may be prohibitively hard --- but instead to a general expression of the transfer function.

\subsection{An expression for the transfer function}
We begin by deriving a particular decomposition of the transfer function of a hBCS as described above by methods inspired by the constructions carried out in \cite{a_TrostorffWaurick23ExponentialStabilityPHS} and \cite{a_WaurickZwart23AsymptoticStabilityPHS} for the study of exponential respectively asymptotic stability.

For any such hBCS and $s \in \complNum$ we have the \emph{fundamental solution} $\Psi_\zeta^s \in \mathbb{K}^{n \times n}$ \emph{of the associated ODE}
\begin{align}
    \label{eq:fundamentalSolutionODE}
        v'(\zeta) = - P_1 ^{-1} \left( P_0(\zeta) - s \mathcal{H}^{-1}(\zeta) \right) v(\zeta).
\end{align}
on $[a,b]$ with $\Psi_a^s = \idMatrix{n}$ as a continuous function $\zeta \mapsto \Psi_\zeta^s$ due to \cite[Cor.~A.2]{a_WaurickZwart23AsymptoticStabilityPHS}.
\begin{prpstn}
\label{prop:transferFunctionExpression}
    The transfer function $\mathbf{G}(s)$ of \pHstand{} satisfying Assumption \ref{ass:systemClassAssumptions} is given for $s \in \rho(A)$ by
    \begin{equation}
    \label{eq:transferFuncExpression}
    \begin{split}
        \mathbf{G}(s) = \widetilde{W}_C \begin{bmatrix}
            \Psi_b^s \\ \idMatrix{n}
        \end{bmatrix} \left( \widetilde{W}_B \begin{bmatrix}
            \Psi_b^s \\ \idMatrix{n}
        \end{bmatrix} \right)^{-1}.
        \end{split}
    \end{equation}
\end{prpstn}
\begin{proof}
    We can write the solution of the first equation from \eqref{eq:transferFunctionODE} using  $\Psi_\zeta^s$ as $x_0 (\zeta) = \Psi_\zeta^s \left( \mathcal{H}x_0 \right)(a)$, so that the boundary conditions become
    \begin{align*}
        u_0 = \widetilde{W}_B \begin{bmatrix}
        \Psi_b^s \\ \idMatrix{n}
    \end{bmatrix} \left( \mathcal{H}x \right)(a), \qquad
    \mathbf{G}(s) u_0 = \widetilde{W}_C \begin{bmatrix}
        \Psi_b^s \\ \idMatrix{n}
    \end{bmatrix} \left( \mathcal{H}x \right)(a),
    \end{align*}
    from which the claim follows by solving for $\mathbf{G}(s)$. The invertibility of $\widetilde{W}_B \left[ \begin{smallmatrix}
            \Psi_b^s \\ \idMatrix{n}
        \end{smallmatrix} \right]$ thereby follows from the existence of a unique $\mathbf{G}(s)$ solving Equation~\eqref{eq:transferFunctionODE}.
\end{proof}
\begin{rmrk}
    Note that for a system with constant coefficients the fundamental solution is given by
    \begin{equation}
    \label{eq:fundamentalSolutionConstantCoefficients}
        \Psi_\zeta^s = e^{- P_1 ^{-1} \left( P_0 - s \mathcal{H}^{-1} \right) (\zeta - a)}.
    \end{equation}
    For a non-constant Hamiltonian for which $P_1 ^{-1} \left( P_0(\xi) - s \mathcal{H}^{-1}(\xi) \right)$  commutes with $\int_a^\zeta P_1 ^{-1} \left( P_0(\tau) - s \mathcal{H}^{-1}(\tau) \right) \intd \tau$ for all $\xi, \zeta \in [a,b]$ we have that
    \begin{equation}
    \label{eq:fundamentalSolutionVaryingCoefficientsCommuting}
        \Psi_\zeta^s = e^{- P_1 ^{-1} \int_a^\zeta \left( P_0 - s \mathcal{H}^{-1} (\tau)\right) \intd \tau}.
    \end{equation}
\end{rmrk}

\subsection{Decomposing $\widetilde{W}_B$}
As $P_1$ is self-adjoint and invertible we have the decomposition $\mathbb{K}^n = E_+ \oplus E_-$ with
\begin{align*}
    E_+ := \left\lbrace v \in \mathbb{K}^n \middle| \exists \lambda > 0 : P_1 v = \lambda v \right\rbrace, \quad
    E_- := \left\lbrace v \in \mathbb{K}^n \middle| \exists \lambda < 0 : P_1 v = \lambda v \right\rbrace.
\end{align*}
Let $\iota_\pm: E_\pm \rightarrow \mathbb{K}^n$ denote the canonical embeddings and define
\begin{align}
\label{eq:defQOperators}
    P_1^\pm := \iota_\pm^* \left(\pm P_1\right) \iota_\pm: E_\pm \rightarrow E_\pm, \qquad
    Q_\pm := \iota_\pm \left( P^\pm_1\right)^\frac{1}{2} \iota^*_\pm: \mathbb{K}^n \rightarrow \mathbb{K}^n,
\end{align}
where the second definition is possible as $P_1^\pm$ are both strictly positive self-adjoint. Note that, as $P_\pm := \iota_\pm \iota_\pm^*: \mathbb{K}^n \rightarrow \mathbb{K}^n$ are the orthogonal projections on $E_\pm$ and since $E_\pm$ are invariant under $P_1$, we have that
\begin{equation*}
    P_1 = Q_+ Q_+ - Q_- Q_- \qquad \andMath \qquad \idMatrix{n} = Q_+ P_1^{-1} Q_+ - Q_- P_1^{-1} Q_-.
\end{equation*}
For the further study of the transfer function as expressed in Equation \eqref{eq:transferFuncExpression}, it will be beneficial if the input matrix $\widetilde{W}_B$ can be decomposed in a particular way with respect to the matrices $Q_+$ and $Q_-$. We believe the following result providing this decomposition to be of interest in its own right.
\begin{prpstn}
\label{prop:decompositionExistence}
Let \pHstand be a hBCS satisfying Assumption \ref{ass:systemClassAssumptions} and let $P_\pm$ and $Q_\pm$ be defined as above. Then the following holds:
\begin{enumerate}
    \item There exist unique matrices $J, L \in \mathbb{K}^{n \times n}$ such that
        \begin{align}
        \label{eq:generalWBdecomp}
            \widetilde{W}_B = \begin{pmatrix}
                J Q_+ - L Q_- & J Q_- - L Q_+
                \end{pmatrix}.
        \end{align}
        In that case $J = \widetilde{W}_B \begin{bmatrix}
            Q_+ \\ -Q_-
        \end{bmatrix} P_1 ^{-1}$ and $L = \widetilde{W}_B \begin{bmatrix}
            Q_- \\ -Q_+
        \end{bmatrix} P_1 ^{-1}$.
    \item The following statements are equivalent:
    \begin{enumerate}
        \item  There exist matrices $K, M \in \mathbb{K}^{n \times n}$ with $K$ invertible such that
        \begin{align}
        \label{eq:inputMatrixDecompositionSpecialLemma}
            \widetilde{W}_B = K \begin{pmatrix}
                Q_+ - M Q_- & Q_- - M Q_+
                \end{pmatrix}.
        \end{align}
        \item The matrix $\widetilde{W}_B \begin{bmatrix}
            P_+ \\ P_-
        \end{bmatrix}$ is invertible.
        \item The operator $A_{\idMatrix{}} := P_1 \partialDer{}{\xi} + P_0 $ with domain
    \begin{align*}
    \dom(A_{\idMatrix{}}) = \set{x_0 \in X \, \middle| \, x_0 \in H^1\left( [a,b], \mathbb{K}^n \right), \widetilde{W}_B \begin{bmatrix}
        x_0(b) \\ x_0(a) 
    \end{bmatrix} = 0 }
    \end{align*}
    is the infinitesimal generator of a $C_0$-semigroup on $\lpSpaceDT{2}{[a,b]}{\mathbb{K}^n}$.
    \item The matrix $J$ from Part 1 is invertible.
    \end{enumerate}
    Furthermore, in this case $K = \widetilde{W}_B \begin{bmatrix}
            Q_+ \\ -Q_-
        \end{bmatrix} P_1 ^{-1}$ and  $M = K^{-1} \widetilde{W}_B \begin{bmatrix}
            Q_- \\ -Q_+
        \end{bmatrix} P_1 ^{-1}$.
\end{enumerate}
\end{prpstn}
\begin{proof}
\begin{enumerate}
    \item Let $\widetilde{W}_B = \begin{pmatrix}
        W_1 & W_2
    \end{pmatrix}$ and define 
    \begin{align*}
        J &:= \widetilde{W}_B \begin{bmatrix}
            Q_+ \\ -Q_-
        \end{bmatrix} P_1 ^{-1} = \left( W_1 Q_+ - W_2 Q_- \right) P_1 ^{-1} \\
        L &:= \widetilde{W}_B \begin{bmatrix}
            Q_- \\ -Q_+
        \end{bmatrix} P_1 ^{-1} = \left( W_1 Q_- - W_2 Q_+ \right) P_1 ^{-1}.
    \end{align*}
    Then $J Q_+ - L Q_- =  W_1 \left( Q_+ P_1 ^{-1} Q_+ - Q_- P_1 ^{-1} Q_- \right) = W_1$ and similarly $J Q_- - L Q_+ =  W_2$, so that $\widetilde{W}_B = \begin{pmatrix}
        J Q_+ - L Q_- & J Q_- - L Q_+
    \end{pmatrix}$.

    Conversely, assume that $\widetilde{W}_B$ is given in the form \eqref{eq:generalWBdecomp}. Then
    \begin{align*}
        \widetilde{W}_B \begin{bmatrix}
            Q_+ \\ Q_-
        \end{bmatrix} P_1 ^{-1} = \left( J Q_+ - L Q_- \right) Q_+ P_1^{-1} - \left( J Q_- - L Q_+ \right) Q_- P_1^{-1} = J
    \end{align*}
    and analogously for $L$, which shows the uniqueness.
    \item By Part 1 of the Lemma and the definitions of $P_\pm$ and $Q_\pm$, clearly the decomposition \eqref{eq:inputMatrixDecompositionSpecialLemma} exists if and only if 
    \begin{equation*}
        J = \widetilde{W}_B \begin{bmatrix}
            Q_+ \\ -Q_-
        \end{bmatrix} P_1 ^{-1} = \widetilde{W}_B \begin{bmatrix}
            P_+ \\ P_-
        \end{bmatrix} \left( Q_+ - Q_- \right) P_1 ^{-1}
    \end{equation*}
    is invertible, giving the equivalence of (a) and (d). Further, as both $Q_+ - Q_-$ and $P_1^{-1}$ are invertible, this is the case if and only if (b) holds.

        The equivalence of (b) and (c) is a consequence of Theorem 1.5 in \cite{a_JacobMorrisZwart15C0Semigroups}.

        Finally, the expressions for $K$ and $M$ follow then directly from Part~(1). \qedhere
\end{enumerate}
\end{proof}
\begin{rmrk}
\label{rem:propertiesDecomposition}
    \begin{enumerate}
        \item Note that statement (c) of Proposition~\ref{prop:decompositionExistence} could equivalently be formulated as the system \pHSmallBrack{n}{P_1}{P_0}{\idMatrix{n}}{\widetilde{W}_B}{\widetilde{W}_C} satisfying Assumption~\ref{ass:systemClassAssumptions}(5).
        \item A direct corollary of Proposition~\ref{prop:decompositionExistence} is that a decomposition \eqref{eq:inputMatrixDecompositionSpecialLemma} exists in particular if $\widetilde{W}_B$ satisfies Equation \eqref{eq:contractionCondition}, as in this case the semigroup generation property is independent of the Hamiltonian. This particular result was first derived by explicit calculation in \cite[Lem.~2.3]{a_TrostorffWaurick23ExponentialStabilityPHS}.
        \item If a decomposition \eqref{eq:inputMatrixDecompositionSpecialLemma} holds true, then by direct calculation we find that
        \begin{align*}
            \widetilde{W}_B R_0^{-1} \Sigma \left( \widetilde{W}_B R_0^{-1} \right)^* = K \left( \idMatrix{n} - M M^\ast \right) K^*,
        \end{align*}
        so that in particular Equation \eqref{eq:contractionCondition} is equivalent to $\idMatrix{n} \geq M M^\ast$ respectively $\matrixNorm{M}{2}{2} \leq 1$. This is the second half of \cite[Lem.~2.3]{a_TrostorffWaurick23ExponentialStabilityPHS}.
        \item An example of a hBCS where no decomposition \eqref{eq:inputMatrixDecompositionSpecialLemma} exists is given by $P_1 = \left[ \begin{smallmatrix}
            0 & 1 \\ 1 & 0
        \end{smallmatrix} \right]$, $P_0 = 0$, $\mathcal{H}(\xi) = \left[ \begin{smallmatrix}
            1 + \xi & 0 \\ 0 & 1
        \end{smallmatrix} \right]$ for $\xi \in [0,1]$ and $\widetilde{W}_B = \left[ \begin{smallmatrix}
            1 & 0 & -1 & 0 \\ 0 & 1 & 0 & 1
        \end{smallmatrix} \right]$. Then by \cite[Ex.~3.2]{a_JacobMorrisZwart15C0Semigroups} the operator $A_{\idMatrix{}}$ does not generate a $C_0$-semigroup and thus by Proposition~\ref{prop:decompositionExistence} no such decomposition exists.
    \end{enumerate}
\end{rmrk}
\noindent For the remainder of this subsection we now take the assumption that the equivalent statements from Proposition~\ref{prop:decompositionExistence} hold true.
\begin{assumption}
\label{ass:decompositionAssumption}
    There exist matrices $K, M \in \mathbb{K}^{n \times n}$ with $K$ invertible such that
    \begin{align}
    \label{eq:inputMatrixDecompositionSpecial}
    \widetilde{W}_B = K \begin{pmatrix}
        Q_+ - M Q_- & Q_- - M Q_+
    \end{pmatrix}.
\end{align}
\end{assumption}
We now consider the middle of the three parts from the expression for $\mathbf{G}(s)$ in Equation \eqref{eq:transferFuncExpression} closer. Note that with the decomposition \eqref{eq:inputMatrixDecompositionSpecial} we have
\begin{align*}
    \widetilde{W}_B \begin{bmatrix}
        \Psi_b^s \\ \idMatrix{n}
    \end{bmatrix} &= K \begin{pmatrix}
        Q_+ - M Q_- & Q_- - M Q_+
    \end{pmatrix} \begin{bmatrix}
        \Psi_b^s \\ \idMatrix{n}
    \end{bmatrix} \\ &= K \left( Q_- + Q_+ \Psi_b^s \right) - K M \left( Q_+ + Q_- \Psi_b^s \right).
\end{align*}

\begin{lmm}
\label{lem:VsInvertibility}
    Let Assumption~\ref{ass:decompositionAssumption} hold for \pHstand. Then the matrix $Q_- + Q_+ \Psi_b^s$ is invertible for any $s \in \rho ( \widetilde{A} )$ where 
    $\widetilde{A}$ is the operator associated to \pHSmallBrack{n}{P_1}{P_0}{\mathcal{H}}{\left[ \begin{smallmatrix}
        Q_+ & Q_-
    \end{smallmatrix} \right]}{\widetilde{W}_C}. In particular there exists $\alpha \in \realNum$ such that $\complNum_\alpha \subset \rho ( \widetilde{A} )$ and hence $Q_- + Q_+ \Psi_b^s$ is invertible for all $s \in \complNum_\alpha$.
\end{lmm}
\begin{proof}
    Note first that $\widetilde{A}$ indeed satisfies Assumption~\ref{ass:systemClassAssumptions}(5), as it generates a $C_0$-semigroup by \cite[Thm.~7.2.4, Thm.~10.3.1]{b_JacobZwart2012} as
    \begin{align*}
        \begin{bmatrix}
        Q_+ & Q_-
    \end{bmatrix} R_0^{-1} \Sigma \left( \begin{bmatrix}
        Q_+ & Q_-
    \end{bmatrix} R_0^{-1} \right)^* = Q_+ P_1^{-1} Q_+ - Q_- P_1^{-1} Q_- = \idMatrix{n} > 0.
    \end{align*}
    By Proposition~\ref{prop:transferFunctionExpression} the transfer function of the system \pHSmallBrack{n}{P_1}{P_0}{\mathcal{H}}{\left[ \begin{smallmatrix}
        Q_+ & Q_-
    \end{smallmatrix} \right]}{\widetilde{W}_C} can thus be represented on $\rho(\widetilde{A})$ as $\widetilde{\mathbf{G}}(s) = \widetilde{W}_C \left[ \begin{smallmatrix}
            \widetilde{\Psi}_b^s \\ \idMatrix{n}
        \end{smallmatrix} \right] \left( \left[ \begin{smallmatrix}
            Q_+ & Q_-
        \end{smallmatrix} \right] \left[ \begin{smallmatrix}
            \widetilde{\Psi}_b^s \\ \idMatrix{n}
        \end{smallmatrix} \right] \right)^{-1}$ where $\widetilde{\Psi}_\zeta^s$ is the fundamental solution associated to that system. Hence in particular $\left[ \begin{smallmatrix}
            Q_+ & Q_-
        \end{smallmatrix} \right] \left[ \begin{smallmatrix}
            \widetilde{\Psi}_b^s \\ \idMatrix{n}
        \end{smallmatrix} \right] = Q_- + Q_+  \widetilde{\Psi}_b^s $ is invertible. But from the definition of the fundamental solution it is clear that $ \widetilde{\Psi}_\zeta^s = \Psi_\zeta^s$ thus proving the claim.
\end{proof}
\begin{rmrk}
    Under the additional assumptions that $\mathcal{H} \in C^1\left( [a,b], \mathbb{K}^{n \times n} \right)$ and $P_0 = - P_0^*$, the operator $\widetilde{A}$ even generates an exponentially stable contraction semigroup \cite[Lem.~9.1.4]{b_JacobZwart2012}. Then there even exists a negative $\alpha < 0$ satisfying the conditions in Lemma \ref{lem:VsInvertibility}. Thus $Q_- + Q_+ \Psi_b^s$ and $(Q_- + Q_+ \Psi_b^s)^{-1}$ are then in particular defined on the imaginary axis and the whole positive right half plane.
\end{rmrk}

\begin{dfntn}
\label{def:VsUsDef}
    For $s \in \complNum_\alpha$ with $\alpha$ from Lemma \ref{lem:VsInvertibility}, define the matrices
    \begin{equation*}
        \mathbf{V}(s) := Q_- + Q_+ \Psi_b^s \qquad \andMath \qquad \mathbf{U}(s) := \left( Q_+ + Q_- \Psi_b^s \right) \mathbf{V}(s)^{-1},
    \end{equation*}
    where $\Psi_\xi^s$ is the fundamental solution associated to the hBCS and $Q_\pm$ are the projection operators defined in Equation \eqref{eq:defQOperators}. 
\end{dfntn}
Using these matrices we can now simplify the expression for the transfer function.

\begin{lmm}
\label{lem:transferFunctionUexpr}
The transfer function of a hBCS satisfying Assumptions \ref{ass:systemClassAssumptions} and \ref{ass:decompositionAssumption} is given on $s \in \rho(A) \cap \complNum_\alpha$ with $\alpha$ from Lemma \ref{lem:VsInvertibility} by $\mathbf{G}(s) = \mathbf{Z}(s) \left( 1 - M \mathbf{U}(s) \right)^{-1} K^{-1},$
    where $\mathbf{Z}(s) = \widetilde{W}_C \begin{bmatrix}
        \Psi_b^s \\ \idMatrix{n}
    \end{bmatrix} \mathbf{V}(s)^{-1}$. This holds in particular on $\complNum_{\widetilde{\alpha}}\subset \rho(A) \cap \complNum_\alpha$ for some $\widetilde{\alpha} \geq \alpha$.
\end{lmm}
\begin{proof}
This follows from direct calculation, as
    \begin{align*}
        \mathbf{G}(s) &= \widetilde{W}_C \begin{bmatrix}
            \Psi_b^s \\ \idMatrix{n}
        \end{bmatrix} \left( \widetilde{W}_B \begin{bmatrix}
            \Psi_b^s \\ \idMatrix{n}
        \end{bmatrix} \right)^{-1} = \widetilde{W}_C \begin{bmatrix}
            \Psi_b^s \\ \idMatrix{n}
        \end{bmatrix} \left( K \mathbf{V}(s) - K M \mathbf{U}(s) \mathbf{V}(s) \right)^{-1} \\ 
        &= \widetilde{W}_C \begin{bmatrix}
            \Psi_b^s \\ \idMatrix{n}
        \end{bmatrix} \mathbf{V}(s)^{-1} \left( \idMatrix{n} - M \mathbf{U}(s) \right)^{-1} K^{-1} 
        = \mathbf{Z}(s) \left( \idMatrix{n} - M \mathbf{U}(s) \right)^{-1} K^{-1} . \qedhere
    \end{align*}
\end{proof}
As mentioned before, one problem in deciding the question of BIBO stability using the characterisation from \cite{a_SchwenningerWierzbaZwart24BIBO} is finding the inverse Laplace transform of $\mathbf{G}$. Even in the case of a constant Hamiltonian where the fundamental solution $\Psi_\zeta^s$ takes the simple form \eqref{eq:fundamentalSolutionConstantCoefficients} of a matrix exponential, we see that due to the inverse $\left( \widetilde{W}_B \left[ \begin{smallmatrix} \Psi_b^s \\ \idMatrix{n} \end{smallmatrix} \right] \right)^{-1}$ it will in general be difficult to find that inverse transform.
 With the representation from Lemma \ref{lem:transferFunctionUexpr} however, this problem can be  circumnavigated in some cases and lead to a sufficient condition for BIBO stability.
\begin{lmm}
\label{lem:NeumannSeriesRewriting}
Let the assumptions of Lemma~\ref{lem:transferFunctionUexpr} hold.
\begin{enumerate}
    \item Let $s \in \complNum_{\widetilde{\alpha}}$ with $\widetilde{\alpha}$ from Lemma~\ref{lem:transferFunctionUexpr}. Then if the series $\sum_{k=0}^\infty \left( M \mathbf{U}(s) \right)^k$ converges we have that
    \begin{align}
    \label{eq:transferFunctionNeumannSum}
        \mathbf{G}(s) = \mathbf{Z}(s) \left( \sum_{k=0}^\infty \left( M \mathbf{U}(s) \right)^k \right) K^{-1}.
    \end{align}
    \item If 
    $\laplaceInvTr{\mathbf{Z}}$ and $\laplaceInvTr{\left( M \mathbf{U} \right)^k}$ exist as measures of bounded total variation for all $k>0$ and in addition
    \begin{equation*}
        \sum_{k=0}^\infty \left\| \laplaceInvTr{\left( M \mathbf{U} \right)^k } \right\|_\mathcal{M} < \infty,
    \end{equation*}
    then \pHSmallBrack{n}{P_1}{P_0}{\mathcal{H}}{\widetilde{W}_B}{\widetilde{W}_C} is BIBO stable and we have 
    \begin{align}
    \label{eq:transferFunctionInvLaplTransNeumannSum}
        \laplaceInvTr{\mathbf{G}} = \laplaceInvTr{\mathbf{Z}} \ast \left( \sum_{k=0}^\infty \laplaceInvTr{\left( M \mathbf{U} \right)^k } \right) K^{-1}.
    \end{align}
\end{enumerate}
\end{lmm}
\begin{proof}
    \begin{enumerate}
        \item Clear.
        \item By the assumptions the right hand side of Equation \eqref{eq:transferFunctionInvLaplTransNeumannSum} is well-defined as a measure of bounded total variation $h \in \mathcal{M} \left( \posRealNum, \mathbb{K}^{n \times n} \right)$. 

        Term-by-term Laplace transforming of the right hand side of Equation \eqref{eq:transferFunctionInvLaplTransNeumannSum} yields the right hand side of Equation \eqref{eq:transferFunctionNeumannSum} and the continuity of the Laplace transform implies that this sum indeed converges to $\laplaceTr{h}$ or equivalently that Equation \eqref{eq:transferFunctionInvLaplTransNeumannSum} holds. BIBO stability then follows from \cite[Thm.~3.4]{a_SchwenningerWierzbaZwart24BIBO}. 
    \end{enumerate}
\end{proof}

\subsection{Equivalent diagonal system}
\label{subsec:diagonalSystem}

For the rest of this section we will now take the following additional assumption.
\begin{assumption}
\label{ass:pHSDiagAssumptions}
    There exist $S, \mathcal{H}^D \in C^1\left( [a,b], \mathbb{K}^{n\times n} \right)$ with $S$ pointwise invertible, $\mathcal{H}^D$ diagonal and positive such that $P_1 \mathcal{H} = S^{-1} P^D_1 \mathcal{H}^D S$, where 
    \begin{align*}
    P_1^D = \begin{pmatrix}
        \idMatrix{m} & 0 \\ 0 & -\idMatrix{n-m}
    \end{pmatrix} 
\end{align*}
for some $0 \leq m \leq n$.
\end{assumption}
\begin{rmrk}
    \begin{enumerate}
        \item It is a well-know result that any hBCS satisfying Assumption \ref{ass:pHSDiagAssumptions} is regular and $L^2$-well-posed \cite[Thm.~13.2.2]{b_JacobZwart2012}.
        \item Using the same splitting as in $P_1^D$ we can write $\mathcal{H}^D = \begin{pmatrix}
        \mathcal{H}^D_+ & 0 \\ 0 & \mathcal{H}^D_-
    \end{pmatrix}.$
Then $\mathcal{H}^D_+$ contains the positive and $-\mathcal{H}^D_-$ the negative eigenvalues of $P_1 \mathcal{H}$ on the respective diagonals. 
Furthermore the matrix $S^{-1}(\xi)$ then has the form
\begin{align*}
    S^{-1}(\xi) = \begin{pmatrix}
        e_1(\xi) & \cdots & e_m(\xi) & e_{m+1}(\xi) & \cdots & e_n(\xi)
    \end{pmatrix}
\end{align*}
with $e_1(\xi), \ldots , e_m(\xi)$ eigenvectors of $P_1 \mathcal{H}(\xi)$ corresponding to the eigenvalues in $\mathcal{H}^D_+(\xi)$ and with $e_{m+1}(\xi), \ldots, e_n(\xi)$ those corresponding to the ones in $-\mathcal{H}^D_-(\xi)$.
    \end{enumerate}
\end{rmrk}

\noindent We now consider the state transformation $\widetilde{x} = S x$. Then the PDE in Equation~\eqref{eq:pHSequation} describing the boundary control system becomes
\begin{align*}
    \partialDer{\widetilde{x}}{t}(\xi, t) 
    &= P_1^D \partialDer{}{\xi} \left( \mathcal{H}^D(\xi) \widetilde{x}(\xi, t) \right) + P_0^D(\xi) \left( \mathcal{H}^D(\xi) \widetilde{x}(\xi, t) \right),
\end{align*}
with $P_0^D (\xi) := S(\xi) \left[ \der{S^{-1}(\xi)}{\xi} + P_0 P_1^{-1} S^{-1} (\xi)  \right] P_1^D$. 

Similarly, the boundary control and observation also change under the state transformation and become
\begin{align*}
    u(t) &= \widetilde{W}_B \begin{bmatrix}
        P_1^{-1} S^{-1}(b) P_1^D & 0 \\
        0 & P_1^{-1} S^{-1}(a) P_1^D
    \end{bmatrix} \begin{bmatrix}
         \left( \mathcal{H}^D \widetilde{x} \right)(b,t) \\  \left( \mathcal{H}^D \widetilde{x} \right)(a,t) 
    \end{bmatrix} = \widetilde{W}_B^D \begin{bmatrix}
         \left( \mathcal{H}^D \widetilde{x} \right)(b,t) \\  \left( \mathcal{H}^D \widetilde{x} \right)(a,t) 
    \end{bmatrix},  \\
    y(t) &= \widetilde{W}_C \begin{bmatrix}
        P_1^{-1} S^{-1}(b) P_1^D & 0 \\
        0 & P_1^{-1} S^{-1}(a) P_1^D
    \end{bmatrix} \begin{bmatrix}
         \left( \mathcal{H}^D \widetilde{x} \right)(b,t) \\  \left( \mathcal{H}^D \widetilde{x} \right)(a,t) 
    \end{bmatrix} = \widetilde{W}_C^D \begin{bmatrix}
         \left( \mathcal{H}^D \widetilde{x} \right)(b,t) \\  \left( \mathcal{H}^D \widetilde{x} \right)(a,t) 
    \end{bmatrix}.
\end{align*}
Hence we find that the diagonalised system is just \pHSmallBrack{n}{P_1^D}{P_0^D}{\mathcal{H}_D}{\widetilde{W}_B^D}{\widetilde{W}_C^D}.
As the system resulting from the state transformation has the same input-output behaviour as the original system, they have in particular the same transfer function and thus one is BIBO stable if and only if the other one is.

However, applying the results from the previous section is much simpler in the diagonalised system due to the special structure of $P_1^D$. We can for example easily calculate the associated matrices 
$Q^D_+ = \left[ \begin{smallmatrix}
    \idMatrix{m} & 0 \\ 0 & 0
\end{smallmatrix} \right]$ 
and 
$Q^D_- = \left[ \begin{smallmatrix}
    0 & 0 \\ 0 & \idMatrix{n-m}
\end{smallmatrix} \right]$ 
used in the expression for the transfer function derived earlier.

Furthermore we find that the transformed input matrix $\widetilde{W}_B^D$ of the diagonalised system always possesses a decomposition of the form \eqref{eq:inputMatrixDecompositionSpecial}.
\begin{prpstn}
\label{prop:decompositionExistenceDiagonal}
    Let \pHSmallBrack{n}{P_1^D}{P_0^D}{\mathcal{H}_D}{\widetilde{W}_B^D}{\widetilde{W}_C^D} be the diagonalised version of a hBCS satisfying Assumptions \ref{ass:systemClassAssumptions} and \ref{ass:pHSDiagAssumptions}. Then the matrix $\widetilde{W}^D_B \begin{bmatrix}
            P^D_+ \\ P^D_-
        \end{bmatrix}$ is invertible and hence there exist matrices $K, M \in \mathbb{K}^{n \times n}$ with $K$ invertible such that
        \begin{align*}
            \widetilde{W}_B^D = K \begin{pmatrix}
                Q_+^D - M Q_-^D & Q_-^D - M Q_+^D
            \end{pmatrix}.
        \end{align*}
\end{prpstn}

\begin{proof}
    Using that $P_1 \mathcal{H} = S^{-1} P_1^D \mathcal{H}^D S$ we find that 
    \begin{align*}
        \widetilde{W}_B^D &= \widetilde{W}_B \begin{bmatrix}
        P_1^{-1} S^{-1}(b) P_1^D & 0 \\
        0 & P_1^{-1} S^{-1}(a) P_1^D
    \end{bmatrix} \\ 
    &= \widetilde{W}_B \begin{bmatrix}
        \mathcal{H}(b) S^{-1}(b) \left( \mathcal{H}^D \right)^{-1}(b) & 0 \\
        0 & \mathcal{H}(a) S^{-1}(a) \left( \mathcal{H}^D \right)^{-1}(a)
    \end{bmatrix}.
    \end{align*}
    Letting $\widetilde{W}_B = \begin{pmatrix}
        W_1 & W_2
    \end{pmatrix}$ we then have
    \begin{align*}
        \widetilde{W}^D_B \begin{bmatrix}
            P^D_+ \\ P^D_-
        \end{bmatrix} &= \widetilde{W}_B \begin{bmatrix}
        \mathcal{H}(b) S^{-1}(b) \left( \mathcal{H}^D \right)^{-1}(b) P^D_+ \\
        \mathcal{H}(a) S^{-1}(a) \left( \mathcal{H}^D \right)^{-1}(a) P^D_-
    \end{bmatrix} \\
    &= W_1 \mathcal{H}(b) S^{-1}(b) \left( \mathcal{H}^D \right)^{-1}(b) P^D_+ + W_2 \mathcal{H}(a) S^{-1}(a) \left( \mathcal{H}^D \right)^{-1}(a) P^D_-.
    \end{align*}
    Considering now the form of $S^{-1}$, $\mathcal{H}^D$ and $P_{\pm}^D$, we see that 
    \begin{align*}
        S^{-1}(b) \left( \mathcal{H}^D \right)^{-1}(b) P^D_+ &= \begin{pmatrix}
        \frac{e_1(b)}{\lambda_1(b)} & \cdots & \frac{e_m(b)}{\lambda_m(b)} & 0 & \cdots & 0
    \end{pmatrix} \\
        S^{-1}(a) \left( \mathcal{H}^D \right)^{-1}(a) P^D_- &= \begin{pmatrix}
        0 & \cdots & 0 & \frac{e_{m+1}(a)}{\lambda_{m+1}(a)} & \cdots & \frac{e_n(a)}{\lambda_n(a)} & 
    \end{pmatrix},
    \end{align*}
    where $\lambda_i(\xi)$ are the entries of $\mathcal{H}^D_\pm(\xi)$ corresponding to the eigenvectors $e_i(\xi)$ of $P_1 \mathcal{H}(\xi)$.

    We thus realise that the range of $\widetilde{W}^D_B  \begin{bmatrix}
            P^D_+ \\ P^D_-
        \end{bmatrix}$ is given by
    \begin{align*}
        W_1 \, \mathcal{H}(b) \, \textrm{span} \left\lbrace e_{i\leq m}(b) \right\rbrace \oplus W_2 \, \mathcal{H}(a) \, \textrm{span} \left\lbrace e_{i> m}(a) \right\rbrace.
    \end{align*}
    But by Theorem 1.5 in \cite{a_JacobMorrisZwart15C0Semigroups} this space is all of $\mathbb{K}^n$ as the operator $A$ generates a $C_0$-semigroup by Assumption \ref{ass:systemClassAssumptions}(5). Hence $\widetilde{W}^D_B \begin{bmatrix}
            P^D_+ \\ P^D_-
        \end{bmatrix}$ is invertible and existence of the decomposition follows from Proposition~\ref{prop:decompositionExistence}.
\end{proof}
\begin{rmrk}
\begin{enumerate}
\item Using the explicit expressions from Proposition~\ref{prop:decompositionExistence}, we find that the matrix $M$ from Proposition \ref{prop:decompositionExistenceDiagonal} is given by
    \begin{equation*}
        M = \left( \widetilde{W}_B \begin{bmatrix}
            P_1^{-1} S^{-1}(b) P^D_+ \\ P_1^{-1} S^{-1}(a) P^D_-
        \end{bmatrix}  \right)^{-1}
     \widetilde{W}_B \begin{bmatrix}
            P_1^{-1} S^{-1}(b) P^D_- \\ P_1^{-1} S^{-1}(a) P^D_+
        \end{bmatrix}.
    \end{equation*}
    \item Note that the proof of Proposition~\ref{prop:decompositionExistenceDiagonal} shows that the existence of the decomposition of $\widetilde{W}^D_B$ is in fact equivalent to Assumption~\ref{ass:systemClassAssumptions}(5).
\end{enumerate}
\end{rmrk}

\subsection{The case $P^D_0 = 0$}
\label{subsec:caseP0D0}
We now want to restrict ourselves to a particular type of hyperbolic boundary control system and thus take the following assumption for the rest of this section.
\begin{assumption}
\label{ass:vanishingP0D}
    For any $\xi \in [a,b]$ we have $P_0^D(\xi) = 0$
\end{assumption}
\begin{rmrk}
    \begin{enumerate}
        \item Note that Assumption \ref{ass:vanishingP0D} is satisfied in particular if the system has vanishing $P_0 = 0$ and a Hamiltonian $\mathcal{H}$ that is constant along $[a,b]$.
        \item There are cases in which $P_0^D = 0$ even though the Hamiltonian is not constant and $P_0 \neq 0$. In fact, $P_0^D$ vanishes if and only if $S^{-1}$ satisfies the differential equation
        $
            \der{S^{-1}(\xi)}{\xi} = - P_0 P_1^{-1} S^{-1} (\xi).
        $
        
        A concrete example on $[a,b] = [0,1]$ where this is the case despite non-constant $\mathcal{H}$ and $P_0 \neq 0$ is given by \pH{2}{\idMatrix{2}}{\left[ \begin{smallmatrix}
            0 & 1 \\ -1 & 0
        \end{smallmatrix} \right]}{\mathcal{H}}{\widetilde{W}_B}{\widetilde{W}_C} with
    \begin{equation*}
        \mathcal{H}(\xi) = \frac{1}{2} \begin{bmatrix}
            2 + 3 \xi - \xi \cos( 2 \xi) & - \xi \sin( 2 \xi ) \\
            - \xi \sin( 2 \xi) & 2 + 3 \xi + \xi \cos( 2 \xi )
        \end{bmatrix}.
    \end{equation*}
    Then we find $S^{ -1}(\xi) = \left[ \begin{smallmatrix}
            \cos(\xi) & - \sin(\xi) \\
            \sin(\xi) & \cos(\xi)
        \end{smallmatrix} \right]$, $\mathcal{H}^D(\xi) = \left[ \begin{smallmatrix}
            1 + \xi & 0 \\ 0 & 1+ 2 \xi
        \end{smallmatrix}\right]$ and  $P_0^D = 0$.
    \end{enumerate}
\end{rmrk}

\noindent The reason for making this assumption is that in this case the diagonalised system takes the simple form
\begin{align*}
    \partialDer{\widetilde{x}}{t}(\xi, t) &= P_1^D \partialDer{}{\xi} \left( \mathcal{H}^D \widetilde{x}(\xi, t) \right), \\
    u(t) &= \widetilde{W}_B^D \begin{bmatrix}
         \left( \mathcal{H}^D \widetilde{x} \right)(b,t) \\  \left( \mathcal{H}^D \widetilde{x} \right)(a,t) 
    \end{bmatrix}, \\
    y(t) &= \widetilde{W}_C^D \begin{bmatrix}
         \left( \mathcal{H}^D \widetilde{x} \right)(b,t) \\  \left( \mathcal{H}^D \widetilde{x} \right)(a,t) 
    \end{bmatrix},
\end{align*}
which allows us to e.g.\ straightforwardly calculate the fundamental solution\footnote{Note that in the following $\Psi_\zeta^s$, $\mathbf{V}(s)$, $\mathbf{U}(s)$, $M$, etc.\ will always refer to the objects calculated for the diagonalised system \pHSmallBrack{n}{P_1^D}{P_0^D}{\mathcal{H}_D}{\widetilde{W}_B^D}{\widetilde{W}_C^D}.} to be
\begin{equation}
\label{eq:fundamentalSolution}
    \Psi_\zeta^s = e^{s \left( P^D_1 \right)^{-1} \int_a^\zeta \left( \mathcal{H}^D \right)^{-1} (\xi) \intd \xi} = \begin{bmatrix}
        e^{s \int_a^\zeta \left( \mathcal{H}^D_+ \right)^{-1} (\xi) \intd \xi} & 0 \\ 0 & e^{- s \int_a^\zeta \left( \mathcal{H}^D_- \right)^{-1} (\xi) \intd \xi}
    \end{bmatrix}.
\end{equation}
Note that this matrix is in particular again diagonal.

As a second step we can now also give explicit expressions for the matrices $\mathbf{V}(s)$, $\mathbf{V}(s)^{-1}$ and $\mathbf{U}(s)$ from Definition \ref{def:VsUsDef} appearing in the expression for the transfer function of the diagonalised system in Lemma \ref{lem:transferFunctionUexpr},
\begin{gather}
\label{eq:diagonalVs}
    \mathbf{V}(s) = \begin{bmatrix}
    e^{s \int_a^b \left( \mathcal{H}^D_+ \right)^{-1} (\xi) \intd \xi} & 0 \\ 0 & \idMatrix{n-m}
\end{bmatrix}, \,\,\,\,\,
    \mathbf{V}(s)^{-1} = \begin{bmatrix}
    e^{-s \int_a^b \left( \mathcal{H}^D_+ \right)^{-1} (\xi) \intd \xi} & 0 \\ 0 & \idMatrix{n-m}
\end{bmatrix},
\\
\label{eq:diagonalUs}
    \mathbf{U}(s) = \begin{bmatrix}
        e^{-s \int_a^b \left( \mathcal{H}^D_+ \right)^{-1} (\xi) \intd \xi} & 0 \\ 0 & e^{- s \int_a^b \left( \mathcal{H}^D_- \right)^{-1} (\xi) \intd \xi}
    \end{bmatrix}.
\end{gather}
Note that, while these expressions are well-defined for all $s \in \overline{\complNum_0}$, the expression for the transfer function found earlier is a-priori still only valid for the domain given there.

We first collect some simple observations on the matrix $\mathbf{U}(s)$.
\begin{lmm}
\label{lem:UsProperties}
Let $\mathbf{U}(s)$ be given as in Equation \eqref{eq:diagonalUs}.
    \begin{enumerate}
        \item $\mathbf{U}(s)$ is a diagonal matrix with entries of the form $e^{-s \lambda_k}$ where $\lambda_k > 0$.
        \item For $s = it$, with $t \in \realNum$, the matrix $\mathbf{U}(s)$ is unitary.
        \item For any $s \in \complNum_0$ the entries in $\mathbf{U}(s)$ are strictly smaller than 1.
        \item $\matrixNorm{\mathbf{U}(s)}{p}{p} \leq 1$ for any $p \in [1,\infty)$ and $s \in \overline{\complNum_0}$ with equality if and only if $s \in i \realNum$.
        \item $\matrixNorm{\mathbf{U}(s)}{p}{p} = \matrixNorm{\mathbf{U}({\realPart{s}})}{p}{p}$ for any $p \in [1,\infty)$ and $s \in \overline{\complNum_0}$.
        \item The map $t \mapsto \matrixNorm{\mathbf{U}(t)}{p}{p}$ is strictly decreasing, $t \in \realNum$, and $\lim_{t \rightarrow \infty} \matrixNorm{\mathbf{U}(t)}{p}{p} = 0$.
        \item $\matrixNorm{\mathbf{U}(s)}{p}{p}$ is strictly decreasing with increasing real part of $s$.
    \end{enumerate}
\end{lmm}
\noindent Let now $\left( M_{ij} \right)_{i,j} = M$ and consider the matrix
\begin{align*}
    M \mathbf{U}(s) = 
    \begin{bmatrix}
        M_{11} e^{-s \lambda_1} & M_{12} e^{-s \lambda_2} & \cdots &  M_{1n} e^{-s \lambda_n} \\
        M_{21} e^{-s \lambda_1} & M_{22} e^{-s \lambda_2} & \cdots &  M_{2n} e^{-s \lambda_n} \\
        \vdots & \vdots & & \vdots \\
        M_{n1} e^{-s \lambda_1} & M_{n2} e^{-s \lambda_2} & \cdots &  M_{nn} e^{-s \lambda_n}
    \end{bmatrix},
\end{align*}
whose form follows from Lemma~\ref{lem:UsProperties}(1). Then the properties of $\mathbf{U}(s)$ from that Lemma imply convergence of the Neumann series $\sum_{k=0}^\infty (M \mathbf{U}(s))^k$ for certain $s$.

\begin{lmm}
\label{lem:convergenceNeumannSeries}
Let $\mathbf{U}(s)$ be given as in Equation~\eqref{eq:diagonalUs} and $M \in \mathbb{K}^{n \times n}$.
    \begin{enumerate}
    \item There exists an $\alpha_0 \geq 0$ such that the series $\sum_{k=0}^\infty (M \mathbf{U}(s))^k$ converges for every $s \in \complNum_{\alpha_0}$. 
    \item If $\matrixNorm{M}{2}{2} \leq 1$ then $\sum_{k=0}^\infty (M \mathbf{U}(s))^k$ converges for every $s \in \complNum_0$.
    \item If $\matrixNorm{M}{2}{2} < 1$ then $\sum_{k=0}^\infty (M \mathbf{U}(s))^k$ converges for every $s \in \overline{\complNum_0}$.
\end{enumerate}
\end{lmm}
\begin{proof} \hphantom{a}
\begin{enumerate}
\item By Parts 5 to 7 of Lemma \ref{lem:UsProperties}, there exists an $\alpha_0 \geq 0$ such that $\matrixNorm{M \mathbf{U}(s)}{2}{2} \leq \matrixNorm{M}{2}{2} \matrixNorm{\mathbf{U}(s)}{2}{2} < 1$ for $s \in \complNum_{\alpha_0}$. Hence the Neumann series $\sum_{k=0}^\infty (M \mathbf{U}(s))^k$ converges for these $s$.

    \item From $\matrixNorm{M}{2}{2} \leq 1$ and $\matrixNorm{\mathbf{U}(s)}{2}{2} < 1$ for $s \in \complNum_0$ we  conclude that $\matrixNorm{M \mathbf{U}(s)}{2}{2} \leq \matrixNorm{\mathbf{U}(s)}{2}{2} < 1$. Hence the Neumann series converges.

    \item Follows directly from $\matrixNorm{M \mathbf{U}(s)}{2}{2} \leq \matrixNorm{M}{2}{2} < 1$ for $\overline{\complNum_0}$. \qedhere
\end{enumerate}   
\end{proof}

\noindent Using Lemma~\ref{lem:convergenceNeumannSeries} and Lemma \ref{lem:transferFunctionUexpr} we then immediately have the following results.
\begin{crllr}
Consider a hBCS satisfying Assumptions \ref{ass:systemClassAssumptions}, \ref{ass:pHSDiagAssumptions} and \ref{ass:vanishingP0D}.
    \begin{enumerate}
        \item There exists $\alpha_0 \geq 0$ such that for $s \in \complNum_{\alpha_0}$
        \begin{equation}
        \label{eq:transferFunctionConvergenceLemma}
            \mathbf{G}(s) = \mathbf{Z}(s) \sum_{k=0}^\infty (M \mathbf{U}(s))^k \, K^{-1}.
        \end{equation}
        \item If $\matrixNorm{M}{2}{2} \leq 1$ then \eqref{eq:transferFunctionConvergenceLemma} holds for $s \in \complNum_0$.
        \item If $\matrixNorm{M}{2}{2} < 1$ then \eqref{eq:transferFunctionConvergenceLemma} holds for $s \in \overline{\complNum_0}$.
    \end{enumerate}
\end{crllr}
By Lemma \ref{lem:NeumannSeriesRewriting}(2) we thus have
\begin{equation}
\label{eq:invLaplaceTransTransfFunctionTermByTerm}
    \laplaceInvTr{\mathbf{G}}(t) = \laplaceInvTr{\mathbf{Z}}(t) \ast \sum_{k=0}^\infty \laplaceInvTr{M \mathbf{U}}(t)^{\ast k} K^{-1}.
\end{equation}
where
\begin{align*}
    \laplaceInvTr{M \mathbf{U}}(t) = \begin{bmatrix}
        M_{11} \delta(t - \lambda_1) & M_{12} \delta(t - \lambda_2) & \cdots &  M_{1n} \delta(t - \lambda_n) \\
        M_{21} \delta(t - \lambda_1) & M_{22} \delta(t - \lambda_2) & \cdots &  M_{2n} \delta(t - \lambda_n) \\
        \vdots & \vdots & & \vdots \\
        M_{n1} \delta(t - \lambda_1) & M_{n2} \delta(t - \lambda_2) & \cdots &  M_{nn} \delta(t - \lambda_n)
    \end{bmatrix}.
\end{align*}

\begin{rmrk}
    Note in passing that, as $\mathbf{Z} \in H^\infty(\overline{\complNum_\alpha})$ and $(M \mathbf{U})^k \in H^\infty(\overline{\complNum_\alpha})$ for $\alpha > \alpha_0$, the uniform convergence implies that also $\mathbf{G} \in H^\infty(\overline{\complNum_\alpha})$. 
\end{rmrk}

\subsection{BIBO stability of the diagonalised system}
We now strive to derive sufficient conditions for BIBO stability using the condition from Part 3 of Lemma \ref{lem:NeumannSeriesRewriting}. By a slight abuse of notation aimed at conciseness, let in the following $\left\| F_{p,q}(s) \right\|_{\mathcal{M}} := \left\| \laplaceInvTr{F_{p,q}} \right\|_{\mathcal{M}}$ denote the total variation of the inverse Laplace transform of the $(p,q)$-th entry of $F$.

With this notation, by Part 3 of Lemma \ref{lem:NeumannSeriesRewriting} it is thus sufficient for BIBO stability to show that for all $1 \leq p,q \leq n$ we have
\begin{align*}
    \left\| \left(\mathbf{Z}\right)_{p,q} \right\|_{\mathcal{M}} < \infty \qquad \andMath \qquad
    \left\| \left( \left( M \mathbf{U}(s) \right)^k\right)_{p,q} \right\|_{\mathcal{M}} < \infty
\end{align*}
for any $k>0$ and that
\begin{equation}
\label{eq:normSumConvergence}
\sum_{k=0}^\infty \left\| \left( \left( M \mathbf{U}(s) \right)^k\right)_{p,q} \right\|_{\mathcal{M}} < \infty.
\end{equation}

\begin{lmm}
    For any $\widetilde{W}_C \in \mathbb{K}^{n \times 2n}$ and $1 \leq p,q \leq n$ we have $\left\| \left(\mathbf{Z}\right)_{p,q} \right\|_{\mathcal{M}} < \infty$.
\end{lmm}
\begin{proof}
    It follows directly from the definition of $\mathbf{Z} = \widetilde{W}_C \begin{bmatrix}
        \Psi_b^s \\ \idMatrix{n}
    \end{bmatrix} \mathbf{V}(s)^{-1}$ and the derived expressions for $\Psi_b^s$ and $\mathbf{V}(s)^{-1}$ that the components of $\mathbf{Z}$ are just finite linear combinations of exponential terms with negative exponent and constants. Thus their inverse Laplace transform will be a finite linear combination of delta distributions and hence a measure of bounded total variation.
\end{proof}

\noindent Considering the structure of $M \mathbf{U}(s)$ we see that any power of it will have the form
\begin{align*}
\left( M \mathbf{U}(s) \right)^k =
    \begin{bmatrix}
        \sum_{i=1}^m \alpha^{(k)}_{11,i} e^{-s \widetilde{\lambda}_i} & \sum_{i=1}^m \alpha^{(k)}_{12,i} e^{-s \widetilde{\lambda}_i} &  \cdots & \sum_{i=1}^m \alpha^{(k)}_{1n,i} e^{-s \widetilde{\lambda}_i} \\
        \sum_{i=1}^m \alpha^{(k)}_{21,i} e^{-s \widetilde{\lambda}_i} & \sum_{i=1}^m \alpha^{(k)}_{22,i} e^{-s \widetilde{\lambda}_i} & \cdots & \sum_{i=1}^m \alpha^{(k)}_{2n,i} e^{-s \widetilde{\lambda}_i} \\
        \vdots & \vdots &  & \vdots \\
        \sum_{i=1}^m \alpha^{(k)}_{n1,i} e^{-s \widetilde{\lambda}_i} & \sum_{i=1}^m \alpha^{(k)}_{n2,i} e^{-s \widetilde{\lambda}_i} & \cdots & \sum_{i=1}^m \alpha^{(k)}_{nn,i} e^{-s \widetilde{\lambda}_i}
    \end{bmatrix}
\end{align*}
with some coefficients $\alpha^{(k)}_{pq,i}$ and pairwise different $\widetilde{\lambda}_i \in \realNum$ with $1 \leq i \leq m$ for some $m \leq n$ and $1 \leq p,q \leq n$. Its inverse Laplace transform is then given by
\begin{align*}
\left( M \mathbf{U}(s) \right)^k =
    \begin{bmatrix}
        \sum_i \alpha^{(k)}_{11,i} \delta\left( t - \widetilde{\lambda}_i \right) & \sum_i \alpha^{(k)}_{12,i} \delta\left( t - \widetilde{\lambda}_i \right) &  \cdots & \sum_i \alpha^{(k)}_{1n,i} \delta\left( t - \widetilde{\lambda}_i \right) \\
        \sum_i \alpha^{(k)}_{21,i} \delta\left( t - \widetilde{\lambda}_i \right) & \sum_i \alpha^{(k)}_{22,i} \delta\left( t - \widetilde{\lambda}_i \right) & \cdots & \sum_i \alpha^{(k)}_{2n,i} \delta\left( t - \widetilde{\lambda}_i \right) \\
        \vdots & \vdots &  & \vdots \\
        \sum_i \alpha^{(k)}_{n1,i} \delta\left( t - \widetilde{\lambda}_i \right) & \sum_i \alpha^{(k)}_{n2,i} \delta\left( t - \widetilde{\lambda}_i \right) & \cdots & \sum_i \alpha^{(k)}_{nn,i} \delta\left( t - \widetilde{\lambda}_i \right)
    \end{bmatrix}
\end{align*}
and hence
$
    \left\| \left( \left( M \mathbf{U}(s) \right)^k \right)_{p,q} \right\|_{\mathcal{M}} = \sum_{i=1}^m \left| \alpha^{(k)}_{pq,i} \right| < \infty$ with $1 \leq p,q \leq n$
always has bounded total variation. For sufficient conditions for BIBO stability it thus remains to guarantee that the sums in Equation~\eqref{eq:normSumConvergence} are all finite. We do this using a number of different additional assumptions on $M$.

\begin{thrm}
\label{thm:sufficientConditions}
Let \pH{n}{P^D_1}{0}{\mathcal{H}^D}{\widetilde{W}^D_B}{\widetilde{W}^D_C} be the diagonalised form of a system satisfying Assumptions~\ref{ass:systemClassAssumptions}, \ref{ass:pHSDiagAssumptions} and \ref{ass:vanishingP0D}. Then if any of the following conditions is satisfied, Equation~\eqref{eq:normSumConvergence} holds and hence the system is BIBO stable.
    \begin{enumerate}
        \item $\matrixNorm{M}{\infty}{\infty} < 1$
        \item $\sum_{k=0}^\infty \left( \absMatrix{M} \right)^k$ converges. This is in particular the case if $\| \absMatrix{M} \| < 1$ in some submultiplicative matrix norm.
        \item There exists a $k_0 > 0$ such that $\displaystyle \sum_{r=1}^n \left\| \left( \left( M \mathbf{U}(s) \right)^{k_0} \right)_{p,r} \right\|_{\mathcal{M}} < 1$ for  $1 \leq p \leq n$.
    \end{enumerate}
\end{thrm}
\begin{proof}
\begin{enumerate}[wide, labelindent=0pt]
\item Denote $m_\infty := \matrixNorm{M}{\infty}{\infty} < 1$. From the form of $M \mathbf{U}(s)$ it then follows in particular that $\sum_r \left\| \left( M \mathbf{U}(s) \right)_{p,r} \right\|_{\mathcal{M}} \leq m_\infty$ for any $p$.

We now claim that  $\sum_r \left\| \left( \left( M \mathbf{U}(s) \right)^k \right)_{p,r} \right\|_{\mathcal{M}} \leq m_\infty^k$ and prove this by induction. By what we just stated this is true for $k = 1$. Let it be true for some $k$ and let $\left( M \mathbf{U}(s) \right)^{k}$ have the form as above. Then
\vspace*{-5pt}
    \begin{align*}
    \left( \left( M \mathbf{U}(s) \right)^{k+1} \right)_{p,q} &= 
    \begin{bmatrix}
        \sum_i \alpha^{(k)}_{p1,i} e^{-s \widetilde{\lambda}_i} & \sum_i \alpha^{(k)}_{p2,i} e^{-s \widetilde{\lambda}_i} & \hspace*{-4pt} \cdots \hspace*{-4pt} & \sum_i \alpha^{(k)}_{pn,i} e^{-s \widetilde{\lambda}_i}
    \end{bmatrix} 
    \begin{bmatrix}
        M_{1q} e^{-s \lambda_q} \\
        M_{2q} e^{-s \lambda_q} \\
        \vdots  \\
        M_{nq} e^{-s \lambda_q} 
    \end{bmatrix} \\
    &= \sum_r \sum_i \alpha^{(k)}_{pr,i} e^{-s \widetilde{\lambda}_i} M_{rq} e^{-s \lambda_q}
    = \sum_r \sum_i \alpha^{(k)}_{pr,i} M_{rq}  e^{-s (\widetilde{\lambda}_i + \lambda_q)}.
\end{align*}
Thus
\begin{align*}
    \left\| \left( \left( M \mathbf{U}(s) \right)^{k+1} \right)_{p,q} \right\|_{\mathcal{M}} \leq \sum_i \left| \sum_r   \alpha^{(k)}_{pr,i} M_{rq} \right| &\leq \sum_i  \sum_r  \left| \alpha^{(k)}_{pr,i} \right| \left| M_{rq} \right| \\ &= \sum_r \left| M_{rq} \right| \left\| \left( \left( M \mathbf{U}(s) \right)^k \right)_{p,r} \right\|_{\mathcal{M}}.
\end{align*}
Hence
\begin{align*}
    &\sum_q \left\| \left( \left( M \mathbf{U}(s) \right)^{k+1} \right)_{p,q} \right\|_{\mathcal{M}} \leq \sum_q \sum_r \left| M_{rq} \right| \left\| \left( \left( M \mathbf{U}(s) \right)^k \right)_{p,r} \right\|_{\mathcal{M}}\\ &\qquad= \sum_r \left\| \left( \left( M \mathbf{U}(s) \right)^k \right)_{p,r} \right\|_{\mathcal{M}} \sum_q \left| M_{rq} \right| \leq m_\infty \sum_r \left\| \left( \left( M \mathbf{U}(s) \right)^k \right)_{p,r} \right\|_{\mathcal{M}} \leq m_\infty^{k+1}.
\end{align*}
This proves the claim and by a Neumann series argument the proposition.

\item For this it is sufficient to show by induction that for any $k \in \natNum$
    \begin{align*}
    \left\| \left( \left( M \mathbf{U}(s) \right)^k \right)_{p,q} \right\|_{\mathcal{M}} \leq \left(\left( \absMatrix{M} \right)^k \right)_{p,q}.
    \end{align*}
    The claim is clear from the expression for $M \mathbf{U}(s)$ for $k=1$. Now suppose the inequality holds for $k \in \natNum$. Then
    \vspace*{-5pt}
    \begin{align*}
    \left( \left( M \mathbf{U}(s) \right)^{k+1} \right)_{p,q} &= 
    \begin{bmatrix}
        \sum_i \alpha^{(k)}_{p1,i} e^{-s \widetilde{\lambda}_i} & \sum_i \alpha^{(k)}_{p2,i} e^{-s \widetilde{\lambda}_i} & \hspace*{-4pt} \cdots \hspace*{-4pt} & \sum_i \alpha^{(k)}_{pn,i} e^{-s \widetilde{\lambda}_i}
    \end{bmatrix} 
    \begin{bmatrix}
        M_{1q} e^{-s \lambda_q} \\
        M_{2q} e^{-s \lambda_q} \\
        \vdots  \\
        M_{nq} e^{-s \lambda_q} 
    \end{bmatrix} \\
    &= \sum_r \sum_i \alpha^{(k)}_{pr,i} e^{-s \widetilde{\lambda}_i} M_{rq} e^{-s \lambda_q}
    = \sum_r \sum_i \alpha^{(k)}_{pr,i} M_{rq}  e^{-s (\widetilde{\lambda}_i + \lambda_q)}.
\end{align*}
Thus
\begin{align*}
    &\left\| \left( \left( M \mathbf{U}(s) \right)^{k+1} \right)_{p,q} \right\|_{\mathcal{M}} \leq \sum_r \sum_i \left| \alpha^{(k)}_{pr,i} M_{rq} \right| = \sum_r \left| M_{rq} \right| \sum_i \left| \alpha^{(k)}_{pr,i} \right| \\ &\qquad= \sum_r \left| M_{rq} \right| \left\| \left( \left( M \mathbf{U}(s) \right)^k \right)_{p,q} \right\|_{\mathcal{M}} \leq \sum_r \left| M_{rq} \right| \left( \left(\absMatrix{M}\right)^{k} \right)_{p,r} = \left( \left(\absMatrix{M}\right)^{k+1} \right)_{p,r}.
\end{align*}

\item Realise that it is sufficient to show that the sum
    \begin{equation*}
        \sum_{k=0}^\infty (M \mathbf{U}(s))^k = \left( \sum_{k=0}^{k_0 - 1} (M \mathbf{U}(s))^k \right) \left( \sum_{\ell=0}^\infty (M \mathbf{U}(s))^{\ell k_0} \right).
    \end{equation*}
    is the Laplace transform of a measure of bounded total variation. The first term is always the Laplace transform of a measure of bounded total variation as a finite sum of such terms. For the second term, let $m_{k_0} := \max_p \sum_r \left\| \left( \left( M \mathbf{U}(s) \right)^{k_0} \right)_{p,r} \right\|_{\mathcal{M}} < 1$. We prove by induction that $\sum_r \left\| \left( \left( M \mathbf{U}(s) \right)^{\ell k_0} \right)_{p,r} \right\|_{\mathcal{M}} \leq m_{k_0}^\ell$. For $\ell = 1$ this is clear. For the induction consider
    \begin{align*}
    &\left( \left( M \mathbf{U}(s) \right)^{(\ell + 1)k_0} \right)_{p,q} \\[-15pt]
    &\hspace*{15pt}= 
    \begin{bmatrix}
        \sum_i \alpha^{(\ell k_0)}_{p1,i} e^{-s \widetilde{\lambda}_i} & \sum_i \alpha^{(\ell k_0)}_{p2,i} e^{-s \widetilde{\lambda}_i} & \hspace*{-4pt} \cdots \hspace*{-4pt} & \sum_i \alpha^{(\ell k_0)}_{pn,i} e^{-s \widetilde{\lambda}_i}
    \end{bmatrix} 
    \begin{bmatrix}
        \sum_j \alpha^{(k_0)}_{1q,j} e^{-s \widetilde{\widetilde{\lambda}}_j} \\
        \sum_j \alpha^{(k_0)}_{2q,j} e^{-s \widetilde{\widetilde{\lambda}}_j} \\
        \vdots  \\
        \sum_j \alpha^{(k_0)}_{nq,j} e^{-s \widetilde{\widetilde{\lambda}}_j} 
    \end{bmatrix} \\
    &\hspace*{15pt}= \sum_r  \left( \sum_i \alpha^{(\ell k_0)}_{pr,i} e^{-s \widetilde{\lambda}_i} \right) \left( \sum_j \alpha^{(k_0)}_{rq,j} e^{-s \widetilde{\widetilde{\lambda}}_j} \right) = \sum_r  \sum_i \sum_j \alpha^{(\ell k_0)}_{pr,i}  \alpha^{(k_0)}_{rq,j} e^{-s (\widetilde{\lambda}_i + \widetilde{\widetilde{\lambda}}_j)} 
\end{align*}
so that
\begin{align*}
    \left\| \left( \left( M \mathbf{U}(s) \right)^{(\ell + 1)k_0} \right)_{p,q} \right\|_\mathcal{M} &\leq 
    \sum_r  \sum_i \sum_j \left| \alpha^{(\ell k_0)}_{pr,i}  \alpha^{(k_0)}_{rq,j} \right| \\ &= \sum_r  \sum_i \left| \alpha^{(\ell k_0)}_{pr,i} \right| \left\| \left( \left( M \mathbf{U}(s) \right)^{k_0} \right)_{r,q} \right\|_{\mathcal{M}} .
\end{align*}
Hence we find
\begin{align*}
    \sum_q \left\| \left( \left( M \mathbf{U}(s) \right)^{(\ell + 1)k_0} \right)_{p,q} \right\|_\mathcal{M} &\leq 
     \sum_r  \sum_i  \left| \alpha^{(\ell k_0)}_{pr,i} \right|m_{k_0} \leq m_{k_0}^{\ell + 1}.
\end{align*}
The statement then follows again from a Neumann series argument. \qedhere
    \end{enumerate}
\end{proof}

\begin{xmpl}
\begin{enumerate}
    \item 
    As an example for a system where BIBO stability can be shown using the sufficient conditions from Theorem~\ref{thm:sufficientConditions}, take the following impedance passive port-Hamiltonian system from \cite[Ex.~9.2.1]{b_JacobZwart2012} of a vibrating string with a damper affixed at one end (Figure~\ref{fig:bdrDampedString}).
    \begin{figure}[h]
    \begin{tikzpicture}[line cap=rect]
        \fill[pattern=north east lines] (4.76, -0.2) rectangle (5.24, -0.35);
        \draw(5,0.8)--(5,0.4);
        \draw(5, .24)--(5, -0.2);
        \draw[thick](4.76, -0.2)--(5.24, -0.2);
        \draw[thick] plot[smooth] coordinates {
            (0.6,0.8) (1.4, 1.2) (2.6, 0.4) (4.2, 1) (5,0.8) };
        \fill[] (4.83,.4) rectangle (5.17,0.3);
        \draw(4.79, .24)--(5.21, .24);
        \draw(4.79, .24)--(4.79, .48);
        \draw(5.21, .24)--(5.21, .48);
    \end{tikzpicture}
    \caption{Vibrating string damped at the boundary}
    \label{fig:bdrDampedString}
    \end{figure}

    Let the control inputs be given by the force $u_1$ applied at the right-hand side and by the velocity $u_2$ at the left-hand side. As observation outputs we choose the velocity $y_1$ at right-hand side and the force acting on the left-hand side. Then we can write this system as a pHS on the interval $[0,1]$ with    
    \begin{gather*}
        P_1 = \begin{bmatrix}
            0 & 1 \\
            1 & 0
        \end{bmatrix}, \quad P_0 = 0, \quad \mathcal{H} = \begin{bmatrix}
            \rho^{-1} & 0 \\
            0 & T
        \end{bmatrix}, \\
        \widetilde{W}_B = \begin{bmatrix}
            k & 1 & 0 & 0 \\
            0 & 0 & 1 & 0
        \end{bmatrix}, \quad 
        \widetilde{W}_C = \begin{bmatrix}
            1 & 0 & 0 & 0 \\
            0 & 0 & 0 & 1
        \end{bmatrix},
    \end{gather*}
    where $\rho$ and $T$ are the mass density and strain of the string, and $k$ is the damping coefficient of the damper. 

    Calculating the component of the transfer function from $u_1$ to $y_1$, we find it to be $\mathbf{G}_{11}(s) = \frac{1}{ k +  \sqrt{\rho T} \coth(\sqrt{\frac{\rho}{T}} s)}$. As no simple expression for the inverse Laplace transform of this function exists, we are again barred from applying \cite[Thm.~3.4]{a_SchwenningerWierzbaZwart24BIBO} directly and thus indeed have to resort to the sufficient conditions derived here.

    As this pHS satisfies Assumption \ref{ass:pHSDiagAssumptions}, we can diagonalise it and find $M = \left[ \begin{smallmatrix}
        0 & \frac{k - \sqrt{\rho T}}{k + \sqrt{\rho T}} \\
        1 & 0
    \end{smallmatrix} \right]$. Now as $\matrixNorm{M}{\infty}{\infty} = \max(1, \frac{k - \sqrt{\rho T}}{k + \sqrt{\rho T}}) = 1 \nless 1$, we cannot apply Theorem~\ref{thm:sufficientConditions}(1). However, we see that $\sum_{k=0}^\infty \left( \absMatrix{M} \right)^k = \frac{1}{2 \sqrt{\rho T}} \left[ \begin{smallmatrix}
         k + \sqrt{\rho T} & k -  \sqrt{\rho T} \\
         k + \sqrt{\rho T} & k + \sqrt{\rho T}
    \end{smallmatrix} \right]$ converges and thus the system is BIBO stabe by Theorem~\ref{thm:sufficientConditions}(2).

    Using Equation~\eqref{eq:invLaplaceTransTransfFunctionTermByTerm} we can then even determine the inverse Laplace transform of $\mathbf{G}_{11}$ explicitly as
    \begin{equation*}
        \laplaceInvTr{\mathbf{G}_{11}} = \frac{1}{k + \sqrt{\rho T}} \delta(\cdot) - \frac{2 \sqrt{\rho T}}{k^2 - \rho T} \sum_{n=1}^\infty \left( \frac{k - \sqrt{\rho T}}{k + \sqrt{\rho T}} \right)^n \delta\left(\cdot - 2 n \sqrt{\frac{\rho}{T}}\right).
    \end{equation*}

    Note that this example further shows that the norm condition from part (1) of Theorem~\ref{thm:sufficientConditions} is only a sufficient but not a necessary condition.

    \item There are systems for which neither the sufficient condition from part (1) nor the one from part (2) of Theorem~\ref{thm:sufficientConditions} are  satisfied, but where part (3) can be applied. 
    
    Consider \pH{2}{\idMatrix{2}}{0}{2 \idMatrix{2}}{\left[\begin{smallmatrix}
        1 & 0 & \frac{1}{2} & \frac{1}{2} \\
        0 & 1 & \frac{1}{2} & - \frac{1}{2}
    \end{smallmatrix}\right]}{\left[\begin{smallmatrix}
        0 & 0 & 1 & 0 \\
        0 & 0 & 0 & 1
    \end{smallmatrix}\right]}.
    Then we find $\mathbf{G}(s) = \begin{bmatrix}
        \frac{1 - 2 e^\frac{s}{2}}{1 - 2 e^s} & \frac{1}{1 - 2 e^s} \\
        \frac{1}{1 - 2 e^s} & - \frac{1 + 2 e^\frac{s}{2}}{1 - 2 e^s}
    \end{bmatrix}$ for which we cannot easily find the inverse Laplace transform and use it to check for BIBO stability.

    Now employing the decomposition from Proposition \ref{prop:decompositionExistenceDiagonal}, we find $M = \frac{1}{2} \left[ \begin{smallmatrix}
        -1 & -1 \\
        -1 & 1
    \end{smallmatrix} \right]$. As $\matrixNorm{M}{\infty}{\infty} = 1$ we cannot apply part (1) and as $\left( \absMatrix{M} \right)^k = \frac{1}{2} \left[ \begin{smallmatrix}
        1 & 1 \\ 1 & 1
    \end{smallmatrix}\right]$ the sum $\sum_{k=0}^\infty \left( \absMatrix{M} \right)^k$ diverges thus ruling out the use of part (2) of \ref{thm:sufficientConditions}.

    However, we find that $(M \mathbf{U}(s))^2 = \frac{e^s}{2} \left[ \begin{smallmatrix}
        1 & 0 \\
        0 & 1
    \end{smallmatrix} \right]$ and thus 
    \begin{equation*}
        \left\| \laplaceInvTr{\left( M \mathbf{U} \right)^2}(t) \right\|_{\mathcal{M}} = \left\| \begin{bmatrix} \frac{\delta(t)}{2} & 0 \\ 0 &  \frac{\delta(t)}{2} \end{bmatrix} \right\|_{\mathcal{M}} = \begin{bmatrix} \frac{1}{2} & 0 \\ 0 &  \frac{1}{2} \end{bmatrix}
    \end{equation*}
    implying by Theorem~\ref{thm:sufficientConditions}(3) that the hBCS is BIBO stable.
\end{enumerate}
\end{xmpl}

\section{Another approach to BIBO}
\label{sec:approachTwoWellPosedness}
Another approach to BIBO stability of hBCS that we want to present here is based upon the $L^1$-well-posedness of these systems when considered on the state space $L^1$.

\begin{lmm}
\label{lem:L1wellPosedness}
    A hBCS satisfying Assumptions~\ref{ass:systemClassAssumptions} and \ref{ass:pHSDiagAssumptions} is $L^1$-well-posed if considered as a system on the state space $L^1$, that is for some (and hence all) $t>0$ there exists a constant $c > 0$ such that for any classical solution $(u,x,y)$ we have that
    \begin{equation*}
        \| y \|_{\lpSpaceDT{1}{[0,t]}{Y}} + \| x(t) \|_{\lpSpaceDT{1}{[a,b]}{\mathbb{K}}} \leq c \left( \| u \|_{\lpSpaceDT{1}{[0,t]}{U}} + \| x(0) \|_{\lpSpaceDT{1}{[a,b]}{\mathbb{K}}} \right).
    \end{equation*}
\end{lmm}
\begin{proof}
    By Assumption~\ref{ass:systemClassAssumptions}(5) and \cite[Thm.~1.5]{a_JacobMorrisZwart15C0Semigroups} the operator $A$ considered as defined analogously on a domain in $\lpSpaceDT{1}{[a,b]}{\mathbb{K}^n}$ generates a $C_0$-semigroup on $\lpSpaceDT{1}{[a,b]}{\mathbb{K}^n}$ and thus by \cite[Thm.~7.1]{a_ZwartGorrecMaschkeVillegas10WellPosedness} the pHS considered on the state space $\lpSpaceDT{1}{[a,b]}{\mathbb{K}^n}$ is $L^1$-well-posed.
\end{proof}

\begin{rmrk}
    \begin{enumerate}
        \item Note that while \cite[Thm.~1.5]{a_JacobMorrisZwart15C0Semigroups} is formulated only for the case of a spatially constant $P_0$ the proof thereof does not make use of this property and it thus also holds in the more general case considered here.
        \item It is essential here to consider the state space $\lpSpaceDT{1}{[a,b]}{\mathbb{K}^n}$, as on any reflexive Banach space (in particular any $L^p$ with $1 < p < \infty$) $L^1$-well-posedness already implies boundedness of the input operator \cite[Thm.~4.2.7]{b_Staffans2005}.
    \end{enumerate}
\end{rmrk}

\begin{thrm}
    If a hyperbolic boundary control system satisfying Assumptions~\ref{ass:systemClassAssumptions} and \ref{ass:pHSDiagAssumptions} is exponentially stable on the state space $\widetilde{X} = \lpSpaceDT{1}{[a,b]}{\mathbb{K}^n}$ then it is BIBO stable. 
\end{thrm}
\begin{proof}
    By Lemma~\ref{lem:L1wellPosedness} we have that the hBCS is $L^1$-well-posed on $\widetilde{X}$.
    But then by \cite[Lem.~4.5.1]{b_Staffans2005} respectively an argument similar to \cite[Lem.~3.22]{a_BombieriEngel14WellPosedness} in the general case of non-compatible systems (see also \cite[Sec.~2.4.4]{th_Wierzba2025}) we can show that the system is $C^\infty$-LILO stable in the sense of \cite[Def.~4.1]{a_SchwenningerWierzba24_ADualNotionToBIBOStability} and thus by \cite[Thm.~4.2 \& Cor.~4.4]{a_SchwenningerWierzba24_ADualNotionToBIBOStability} also BIBO stable.
\end{proof}

While this theorem provides us with a sufficient condition for BIBO stability it leaves us with the problem of having to determine whether the semigroup generated by the operator $A$ on $L^1$ is exponentially stable. While it is easy to e.g.\ show that the spectrum of the operator $A$ considered on $L^1$ agrees with the one it has on $L^2$, it seems unclear so far which connection (if any) there is between exponential stability on $L^2$ and exponential stability on $L^1$. 

Furthermore, many approaches to checking for exponential stability, such as the Gearhart-Prüss theorem --- which 
forms the basis of the study of exponential stability of the systems in \cite{a_TrostorffWaurick23ExponentialStabilityPHS} --- cannot be applied in the $L^1$ case as they only hold on Hilbert spaces.

Results of a general nature implying exponential stability of a system on $L^1$ seem to exist only in  restricted cases such as e.g.\ for positive semigoups \cite[Thm.~3.5.3]{b_vanNeerven96_AsymptoticBehaviour}.

\section{Conclusion}
\subsection{Some conjectures}
During our studies of BIBO stability for the 1-D hyperbolic boundary control systems underlying this work, a few observations have been made that suggest certain conjectures.

First, we have so far not encountered a system that satisfies the condition \eqref{eq:contractionCondition} as a strict inequality and is not BIBO stable. Considering how \eqref{eq:contractionCondition} relates to the norm of $M$ and the role this matrix plays in the construction of the counterexample in Proposition \ref{prop:existenceNonBIBOpHS} as well as in the arguments in \cite[Thm.~5.6]{a_TrostorffWaurick23ExponentialStabilityPHS}, motivates us to propose the following conjecture on a sufficient condition for BIBO stability.
\begin{cnjctr}
\label{conj:BIBOconjecturePositive}
    If $\widetilde{W}_B R_0^{-1} \Sigma \left( \widetilde{W}_B R_0^{-1} \right)^* > 0$ then the hBCS is BIBO stable.
\end{cnjctr}
Note that by Remark \ref{rem:propertiesDecomposition} this condition is --- in the case that Assumption \ref{ass:decompositionAssumption} holds --- equivalent to $\| M \|_{\ell^2 \rightarrow \ell^2} < 1$.

That the condition from this conjecture, could only be a sufficient condition for BIBO stability follows from the following example, which is based upon a class of systems discussed in other contexts in \cite[Sec.~5]{a_Engel13GeneratorProperty} and \cite[Ex.~4.4]{a_TrostorffWaurick23ExponentialStabilityPHS}.
\begin{xmpl}
    Consider the system \pHSmallBrack{2}{\idMatrix{2}}{0}{\idMatrix{2}}{\widetilde{W}_B}{\widetilde{W}_C}  shown to be BIBO stable in the proof of Proposition \ref{prop:BIBOdepHamiltonian}, but for which $
       \widetilde{W}_B R_0^{-1} \Sigma \left( \widetilde{W}_B R_0^{-1} \right)^*  \ngtr 0
    $.
\end{xmpl}
One potential way to approach Conjecture~\ref{conj:BIBOconjecturePositive} --- at least in the case of the diagonalised system and $P_0^D = 0$ --- may be via Theorem~\ref{thm:sufficientConditions}(3). Indeed, we have not yet found an example for a system with $\| M \|_{\ell^2 \rightarrow \ell^2} < 1$ for which no $k_0$ as required in this proposition exists. This motivates us to propose the following conjecture which would imply Conjecture \ref{conj:BIBOconjecturePositive} for this subclass of hBCS.
\begin{cnjctr}
\label{conj:MnormHigherOrdBoundedness}
    Let $\| M \|_{\ell^2 \rightarrow \ell^2} < 1$. Then there exists a $k_0$ such that 
    \begin{equation*}
        \sum_r \left\| \left( \left( M \mathbf{U}(s) \right)^{k_0} \right)_{p,r} \right\|_{\mathcal{M}} < 1, \qquad 1 \leq p \leq n.
    \end{equation*}
\end{cnjctr}
Note that the strict inequality is necessary here, as for the case $\matrixNorm{M}{2}{2} = 1$ there are counterexamples.
\begin{xmpl}
    Consider \pHSmallBrack{2}{\idMatrix{2}}{0}{\left[ \begin{smallmatrix}
        2 & 0 \\ 0 & 1
    \end{smallmatrix} \right]}{\widetilde{W}_B}{\widetilde{W}_C} with $
    \widetilde{W}_B = \left[ \begin{smallmatrix}
            1 & 0 & \frac{1}{2} & \frac{1}{2} \\ 0 & 1 & -\frac{1}{2} & -\frac{1}{2}    \end{smallmatrix} \right]$ and $W_C = \left[ \begin{smallmatrix}
            0 & 0 & 1 & 0 \vphantom{\frac{1}{2}}\\
            0 & 0 & 0 & 1 \vphantom{\frac{1}{2}}
    \end{smallmatrix} \right]$.
    Then we find $M = \frac{1}{2} \left[ \begin{smallmatrix}
        -1 & -1 \\ 1 & 1
    \end{smallmatrix} \right]$ and thus $M \mathbf{U}(s) = \frac{1}{2} \left[ \begin{smallmatrix}
        -e^{-2s} & -e^{-s} \\ e^{-2s} & e^{-s}
    \end{smallmatrix} \right].$
    By induction we can then show that
    \begin{equation*}
        \left( M \mathbf{U}(s) \right)^k = \frac{(-1)^k}{2^k} \sum_{i = 0}^{k-1} \binom{k-1}{i} e^{-s(2k - i)} (-1)^i \begin{bmatrix}
            1 & e^s \\
            -1 & - e^s
        \end{bmatrix}.
    \end{equation*}
    Therefore we find
    $\left\| \left( M \mathbf{U}(s) \right)^k \right\|_\mathcal{M} = \frac{1}{2^k} \sum_{i = 0}^{k-1} \binom{k-1}{i}   \begin{bmatrix}
            1 & 1 \\
            1 & 1
        \end{bmatrix} = \frac{1}{2} \begin{bmatrix}
            1 & 1 \\
            1 & 1
        \end{bmatrix},$
    so that $\sum_r \left\| \left( \left( M \mathbf{U}(s) \right)^k \right)_{p,r} \right\|_{\mathcal{M}} = 1$ for all $k$. 
\end{xmpl}
Similar, we have so far not even yet encountered an exponentially stable hBCS that is not also BIBO stable. This may motivate one to suggest the following --- due to \cite[Lem.~9.1.4]{b_JacobZwart2012} in some sense more general --- conjecture.
\begin{cnjctr}
    Any exponentially stable hBCS satisfying Assumption \ref{ass:systemClassAssumptions} is BIBO stable.
\end{cnjctr}

\subsection{Outlook}
While the final result of the approach discussed in Section~\ref{sec:approachOneTransferFunction} only provides conditions for a restricted case $P^D_0 = 0$ and thus in particular for the case of $P_0 = 0$ and a constant Hamiltonian, one could hope that these can be straightforwardly extended to the more general case using simple perturbation arguments as is e.g.\ done in \cite{b_JacobZwart2012} to show $L^2$-well-posedness for pHS or in \cite{a_SchwenningerWierzbaZwart24BIBO} to show BIBO stability of certain parabolic systems. Unfortunately there are convincing reasons that this does not work analogously in this case, such as the observation that the simple shift-semigroup with identity input is not BIBO stable \cite[Ex.~3.1]{a_SchwenningerWierzba24_ADualNotionToBIBOStability}.

Besides the extension of the results to the case $P^D_0 \neq 0$ there are two further evident routes for possible generalisation. The first would be to consider higher order hBCS such as they were for example investigated for asymptotic stability in \cite{a_WaurickZwart23AsymptoticStabilityPHS} employing an approach also based upon an adapted concept of an associated fundamental solution. The second lies in the extension to higher spatial dimensions as was studied e.g.\ in \cite{a_vanDerSchaftMaschke02DistributedPHS, th_Skrepek2021}. Here, however, one has to note that the input and output spaces may be infinite-dimensional --- as with boundary control along a one-dimensional boundary for instance. In this case the characterisation of BIBO stability by the transfer function employed in the first approach is no longer valid. Hence in this case only the second approach using $L^1$-well-posedness may lead to sufficient conditions. 

\section*{Acknowledgements}
The authors would like to thank Prof. M. Waurick (TU Freiberg) for helpful discussions on this subject.


\begin{thebibliography}{10}

\bibitem{a_AroraEtAl24SemiUniformStability}
S.~Arora, F.~Schwenninger, I.~Vukusic, and M.~Waurick.
\newblock A universal example for quantitative semi-uniform stability, arXiv:2410.02357 [math.AP], 2024.

\bibitem{th_Augner2016}
{Augner, Bj\"orn}.
\newblock {\em {Stabilisation of infinite-dimensional port-{Hamiltonian} systems via dissipative boundary feedback}}.
\newblock {Thesis}, Bergische Universit{ä}t Wuppertal, {2016}.

\bibitem{b_Bartecki16_ModelingAnalysisOfLinHypSystemsBalanceLaws}
K.~Bartecki.
\newblock {\em Modeling and analysis of linear hyperbolic systems of balance laws}, volume~48 of {\em Stud. Syst. Decis. Control}.
\newblock Cham: Springer, 2016.

\bibitem{a_AbusaksakaPartingtonBIBO2013}
A.~{Bashar Abusaksaka} and J.~R. Partington.
\newblock {BIBO} stability of some classes of delay systems and fractional systems.
\newblock {\em Systems \& Control Letters}, 64:43--46, 2014.

\bibitem{b_BastinCoron16_Stability1dHyperbolicSystems}
G.~Bastin and J.-M. Coron.
\newblock {\em {Stability and boundary stabilization of 1-D hyperbolic systems}}, volume~88 of {\em Progress in Nonlinear Differential Equations and their Applications}.
\newblock Aug. 2016.

\bibitem{a_BastinCoronHayat21ISSwrtSupNorms}
G.~Bastin, J.-M. Coron, and A.~Hayat.
\newblock Input-to-state stability in sup norms for hyperbolic systems with boundary disturbances.
\newblock {\em Nonlinear Analysis}, 208:112300, 2021.

\bibitem{a_BergerPucheSchwenningerFunnel2020}
T.~Berger, M.~Puche, and F.~L. Schwenninger.
\newblock Funnel control in the presence of infinite-dimensional internal dynamics.
\newblock {\em Systems \& Control Letters}, 139:104678, 2020.

\bibitem{a_BombieriEngel14WellPosedness}
M.~Bombieri and K.-J. Engel.
\newblock A semigroup characterization of well-posed linear control systems.
\newblock {\em Semigroup Forum}, 88(2):366--396, 2014.

\bibitem{a_CallierDesoer1978}
F.~Callier and C.~Desoer.
\newblock An algebra of transfer functions for distributed linear time-invariant systems.
\newblock {\em IEEE Transactions on Circuits and Systems}, 25(9):651--662, 1978.

\bibitem{a_WangCobb03BIBOTimeVar}
D.~Cobb and C.-J. Wang.
\newblock A characterization of bounded-input bounded-output stability for linear time-varying systems with distributional inputs.
\newblock {\em SIAM Journal on Control and Optimization}, 42(4):1222--1243, 2003.

\bibitem{b_Dafermos16_HyperbolicConservationLawsInContPhys}
C.~M. Dafermos.
\newblock {\em Hyperbolic conservation laws in continuum physics}, volume 325 of {\em Grundlehren Math. Wiss.}
\newblock Berlin: Springer, 2016.

\bibitem{b_Duindam09RedPHSBook}
V.~Duindam, A.~Macchelli, S.~Stramigioli, and H.~Bruyninckx.
\newblock {\em Modeling and Control of Complex Physical Systems: The Port-{Hamiltonian} approach}.
\newblock Springer, Berlin, Heidelberg, 2009.

\bibitem{a_Engel13GeneratorProperty}
K.-J. Engel.
\newblock Generator property and stability for generalized difference operators.
\newblock {\em Journal of Evolution Equations}, 13(2):311--334, 2013.

\bibitem{b_GripenbergLondenStaffans1990}
G.~Gripenberg, S.-O. Londen, and O.~Staffans.
\newblock {\em Volterra {Integral} and {Functional} {Equations}}.
\newblock Cambridge University Press, Cambridge, 1990.

\bibitem{a_HastirEtAl23NonlinearBIBO}
A.~Hastir, R.~Hosfeld, F.~L. Schwenninger, and A.~A. Wierzba.
\newblock {BIBO} stability for funnel control: Semilinear internal dynamics with unbounded input and output operators.
\newblock In F.~L. Schwenninger and M.~Waurick, editors, {\em Systems Theory and PDEs}, pages 189--217, Cham, 2024. Birkh{ä}user.

\bibitem{a_Hayat21BoundaryStabilization}
A.~Hayat.
\newblock Boundary stabilization of 1d hyperbolic systems.
\newblock {\em Annual Reviews in Control}, 52:222 – 242, 2021.

\bibitem{a_IlchmannRyanTrennFunnel2005}
A.~Ilchmann, E.~P. Ryan, and S.~Trenn.
\newblock Tracking control: Performance funnels and prescribed transient behaviour.
\newblock {\em Systems \& Control Letters}, 54(7):655--670, 2005.

\bibitem{a_JacobMorrisZwart15C0Semigroups}
B.~Jacob, K.~Morris, and H.~Zwart.
\newblock ${C}_0$-semigroups for hyperbolic partial differential equations on a one-dimensional spatial domain.
\newblock {\em Journal of Evolution Equations}, 15(2):493--502, 2015.

\bibitem{a_JacobZwart18OperatorTheoreticApproach}
B.~Jacob and H.~Zwart.
\newblock An operator theoretic approach to infinite-dimensional control systems.
\newblock {\em GAMM Mitteilungen}, 41(4), 2018.

\bibitem{a_JacobZwart23_InfiniteDimPHS}
B.~Jacob and H.~Zwart.
\newblock Infinite-dimensional linear port-hamiltonian systems on a one-dimensional spatial domain: An introduction, arXiv:2308.01822 [math.AP], 2023.

\bibitem{b_JacobZwart2012}
B.~Jacob and H.~J. Zwart.
\newblock {\em Linear Port-{Hamiltonian} Systems on Infinite-dimensional Spaces}.
\newblock Springer Basel, Basel, 2012.

\bibitem{a_leGorrecZwartMaschke05DiracStructuresBCS}
Y.~Le~Gorrec, H.~Zwart, and B.~Maschke.
\newblock Dirac structures and boundary control systems associated with skew-symmetric differential operators.
\newblock {\em SIAM Journal on Control and Optimization}, 44(5):1864--1892, 2005.

\bibitem{a_MacchelliStramigioliMelchiorri06ManipulatorsFlexibleLinks}
A.~Macchelli, S.~Stramigioli, and C.~Melchiorri.
\newblock Port-based modelling of manipulators with flexible links.
\newblock In {\em Proceedings 2006 IEEE International Conference on Robotics and Automation, ICRA 2006}, pages 1886--1891, United States, 2006. IEEE.

\bibitem{th_Mattioni2021}
A.~Mattioni.
\newblock {\em {Modelling and {Stability} {Analysis} of {Flexible} {Robots} : a distributed parameter port-{Hamiltonian} approach}}.
\newblock Thesis, {Universit{\'e} Bourgogne Franche-Comt{\'e}}, Apr. 2021.

\bibitem{b_vanNeerven96_AsymptoticBehaviour}
J.~v. Neerven.
\newblock {\em The asymptotic behaviour of semigroups of linear operators}.
\newblock Operator theory, advances and applications; vol. 88. Birkh{ä}user, Basel, 1996.

\bibitem{a_PasumarthyVanDerschaft06CanalSystems}
R.~Pasumarthy and A.~Schaft.
\newblock A port-{Hamiltonian} approach to modeling and interconnections of canal systems.
\newblock In {\em Proceedings of the 17th International Symposium on Mathematical Theory of Networks and Systems, 2006}, Kyoto, Japan, 2006.

\bibitem{a_RashadEtAl20DistributedPHS}
R.~Rashad, F.~Califano, A.~J. van~der Schaft, and S.~Stramigioli.
\newblock {Twenty years of distributed port-{Hamiltonian} systems: a literature review}.
\newblock {\em IMA Journal of Mathematical Control and Information}, 37(4):1400--1422, 2020.

\bibitem{a_Schwenninger20ISSboundaryControl}
F.~L. Schwenninger.
\newblock Input-to-state stability for parabolic boundary control: linear and semilinear systems.
\newblock In J.~Kerner, H.~Laasri, and D.~Mugnolo, editors, {\em Control Theory of Infinite-Dimensional Systems}, pages 83--116, Cham, 2020. Birkh{ä}user.

\bibitem{a_SchwenningerWierzba24_ADualNotionToBIBOStability}
F.~L. Schwenninger and A.~A. Wierzba.
\newblock A dual notion to {BIBO} stability.
\newblock {\em Journal of Mathematical Analysis and Applications}, 545(2):129156, 2025.

\bibitem{a_SchwenningerWierzbaZwart24BIBO}
F.~L. Schwenninger, A.~A. Wierzba, and H.~Zwart.
\newblock On {BIBO} stability of infinite-dimensional linear state-space systems.
\newblock {\em SIAM Journal on Control and Optimization}, 62(1):22--41, 2024.

\bibitem{b_Serre99_SystemsOfConservationLaws}
D.~Serre.
\newblock {\em Systems of Conservation Laws 1: Hyperbolicity, Entropies, Shock Waves}.
\newblock Cambridge University Press, 1999.

\bibitem{th_Skrepek2021}
N.~Skrepek.
\newblock {\em Linear port-{Hamiltonian} Systems on Multidimensional Spatial Domains}.
\newblock Thesis, Bergische Universit{ä}t Wuppertal, 2021.

\bibitem{b_Staffans2005}
O.~Staffans.
\newblock {\em Well-Posed Linear Systems}.
\newblock Encyclopedia of Mathematics and its Applications. Cambridge University Press, Cambridge, 2005.

\bibitem{a_TrostorffWaurick23ExponentialStabilityPHS}
S.~Trostorff and M.~Waurick.
\newblock Characterisation for exponential stability of port-{Hamiltonian} systems, arXiv:2201.10367 [math.AP], 2024.

\bibitem{b_TucsnakWeiss2009}
M.~Tucsnak and G.~Weiss.
\newblock {\em Observation and Control for Operator Semigroups}.
\newblock Birkh{ä}user Basel, Basel, 2009.

\bibitem{a_UnserNoteOnBiboStability}
M.~Unser.
\newblock A note on {BIBO} stability.
\newblock {\em IEEE Transactions on Signal Processing}, 68:5904--5913, 2020.

\bibitem{a_vanDerSchaft06pHSIntroduction}
A.~{van der Schaft}.
\newblock Port-{Hamiltonian} systems: an introductory survey.
\newblock In M.~Sanz-Sole, J.~Soria, J.~Varona, and J.~Verdera, editors, {\em Proceedings of the International Congress of Mathematicians Vol. III}, pages 1339--1365, Z{ü}rich, 2006. European Mathematical Society Publishing House.

\bibitem{a_vanDerSchaft20PHSModeling}
A.~van~der Schaft.
\newblock Port-{Hamiltonian} modeling for control.
\newblock {\em Annual Review of Control, Robotics, and Autonomous Systems}, 3:393--416, 2020.

\bibitem{a_vanDerSchaftMaschke02DistributedPHS}
A.~{van der Schaft} and B.~Maschke.
\newblock Hamiltonian formulation of distributed-parameter systems with boundary energy flow.
\newblock {\em Journal of Geometry and Physics}, 42(1):166--194, 2002.

\bibitem{th_Villegas2007}
J.~A. Villegas.
\newblock {\em A Port-{Hamiltonian} Approach to Distributed Parameter Systems}.
\newblock Thesis, University of Twente, 2007.

\bibitem{a_WangCobb96BIBOTimeInv}
C.-J. Wang and J.~D. Cobb.
\newblock A characterization of bounded-input bounded-output stability for linear time-invariant systems with distributional inputs.
\newblock {\em SIAM Journal on Control and Optimization}, 34(3):987--1000, 1996.

\bibitem{a_WaurickZwart23AsymptoticStabilityPHS}
M.~Waurick and H.~Zwart.
\newblock Asymptotic stability of port-{Hamiltonian} systems.
\newblock In F.~L. Schwenninger and M.~Waurick, editors, {\em Systems Theory and PDEs}, pages 91--122, Cham, 2024. Birkh{ä}user.

\bibitem{th_Wierzba2025}
{Wierzba, Alexander A.}
\newblock {\em {On BIBO stability of infinite-dimensional systems}}.
\newblock {Thesis}, University of Twente, {2025}.

\bibitem{a_ZhouEtAl15DistPHSReactionDiffusion}
W.~Zhou, B.~Hamroun, Y.~L. Gorrec, and F.~Couenne.
\newblock Infinite dimensional port-{Hamiltonian} representation of reaction diffusion processes.
\newblock {\em IFAC-PapersOnLine}, 48(1):476--481, 2015.

\bibitem{a_ZhouEtAl12DistPHSChemicalReactors}
W.~Zhou, B.~Hamroun, Y.~{Le Gorrec}, and F.~Couenne.
\newblock Infinite dimensional port-{Hamiltonian} representation of chemical reactors.
\newblock {\em IFAC Proceedings Volumes}, 45(19):248--253, 2012.

\bibitem{a_ZwartGorrecMaschkeVillegas10WellPosedness}
H.~Zwart, Y.~{Le Gorrec}, B.~Maschke, and J.~Villegas.
\newblock Well-posedness and regularity of hyperbolic boundary control systems on a one-dimensional spatial domain.
\newblock {\em ESAIM: COCV}, 16(4):1077--1093, 2010.

\end{thebibliography}
\end{document}